\documentclass[onefignum,onetabnum]{siamart190516}

\usepackage{amsfonts}
\usepackage{graphicx}
\usepackage{epstopdf}
\usepackage{mathabx}
\usepackage{enumitem}
\usepackage{bbm}
\usepackage{multirow}
\usepackage{booktabs}
\usepackage{siunitx}

\newcommand*\samethanks[1][\value{footnote}]{\footnotemark[#1]}

\newcommand{\R}{\mathbb{R}}
\newcommand{\D}{\mathcal{D}}

\newcommand{\W}{\mathcal{W}}

\renewcommand{\P}{\mathbb{P}}

\newcommand{\Lip}{\mathrm{Lip}}

\newcommand{\F}{\mathcal{F}}
\newcommand{\G}{\mathcal{G}}
\newcommand{\E}{\mathbb{E}}
\newcommand{\T}{\mathsf{T}}
\newcommand{\m}{\hspace{0.25mm}}
\newcommand{\mm}{\hspace{5mm}}
\newcommand{\n}{\hspace{-0.25mm}}
\DeclareMathOperator*{\mesh}{mesh}

\DeclareMathOperator*{\cov}{Cov}

\usepackage{algorithmic}
\ifpdf
  \DeclareGraphicsExtensions{.eps,.pdf,.png,.jpg}
\else
  \DeclareGraphicsExtensions{.eps}
\fi

\newsiamremark{remark}{Remark}
\newsiamremark{hypothesis}{Hypothesis}
\crefname{hypothesis}{Hypothesis}{Hypotheses}
\newsiamthm{claim}{Claim}
\newsiamthm{example}{Example}

\headers{Convergence of adaptive approximations for SDEs}{J.~Foster and A.~Jelin\v{c}i\v{c}}

\title{On the convergence of adaptive approximations\\ for stochastic differential equations}

\author{James Foster\m\thanks{\m University of Bath, Department of Mathematical Sciences. Email: \texttt{\{jmf68, aj2382\}@bath.ac.uk}}
\and
Andra\v{z} Jelin\v{c}i\v{c}\m\samethanks
}

\usepackage{amsopn}

\ifpdf
\hypersetup{
  pdftitle={On the convergence of adaptive approximations for SDEs},
  pdfauthor={James Foster and Andraž Jelinčič}
}
\fi

\begin{document}

\maketitle

% REQUIRED
\begin{abstract}
In this paper, we study numerical approximations for stochastic differential equations (SDEs) that use \textit{adaptive} step sizes. In particular, we consider a general setting where decisions to reduce step sizes are allowed to depend on the future trajectory of the underlying Brownian motion. Since these adaptive step sizes may not be previsible, the standard mean squared error analysis cannot be directly applied to show that the numerical method converges to the solution of the SDE.
Building upon the pioneering work of Gaines and Lyons, we instead use \textit{rough path theory} to establish pathwise convergence for a wide class of adaptive numerical methods on general Stratonovich SDEs (with sufficiently smooth vector fields). To our knowledge, this is the first convergence guarantee applicable to standard solvers, such as the Milstein and Heun methods, with non-previsible step sizes. In our analysis, we require adaptive step sizes to have a ``no skip'' property and to take values at only dyadic times. Secondly, in contrast to the Euler-Maruyama method, we require the SDE solver to have unbiased L\'{e}vy area terms in its Taylor expansion. We conjecture that for adaptive SDE solvers more generally, convergence is still possible provided the scheme does not introduce ``L\'{e}vy area bias''.
We present a simple example where the step size control can skip over previously considered times, resulting in the numerical method converging to an incorrect limit (i.e.~not the Stratonovich SDE).
Finally, we conclude with an experiment demonstrating the accuracy of Heun's method and a newly introduced Splitting Path-based Runge-Kutta scheme (SPaRK) when used with adaptive step sizes.
\end{abstract}

% REQUIRED

% REQUIRED

\section{Introduction}
Stochastic differential equations (SDEs) have seen widespread use in the physical, engineering and social sciences for describing random systems. Examples of SDEs can be found within finance \cite{brigo2006finance, oksendal2003SDEs, shreve2004finance},  biology \cite{allen2010biology, browning2020biology, jha2012biology, vadillo2019LV}, physics \cite{strauss2017physics, leimkuhler2015ULD, milstein2004physics, sobczyk1991physics} and, more recently, in data science \cite{chen2014SGHMC, cheng2018MCMC, dockhorn2022diffusion, foster2021shifted, salvi2023NSDEs, kidger2021NSDEs1, li2020NSDEs, li2019LangevinMC, song2021scoredbased, welling2011SGLD}.
However, just as for ordinary differential equations (ODEs), solutions of SDEs are rarely obtainable in closed form, and numerical methods are often required in practice.\smallbreak

In such applications, this is typically done through Monte Carlo simulation, where numerical solutions are computed using independently generated random variables
(for example, corresponding to the increments of the underlying Brownian motion).
By independently generating multiple numerical solutions, one can then use a Monte Carlo estimator to compute quantities relating to the average behaviour of the SDE.
As motivation for this paper, we note that the \textit{status quo} is to propagate numerical solutions of SDEs forwards in time using a fixed step size, whereas so-called ``adaptive'' step sizes have seen great success for discretizing ODEs (see \cite[Section II.4]{hairer1993odes} and \cite{soderlind2002odes}).
We refer the reader to the book \cite{higham2021SDEs} for an accessible introduction to SDE numerics.\vspace{-1.5mm}

\begin{figure}[!hbt]
    \centering
    \includegraphics[width=\textwidth]{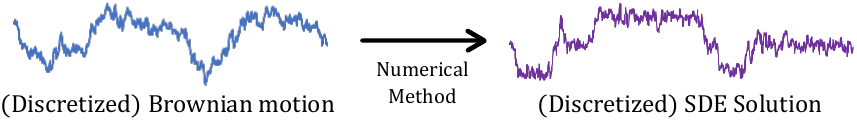}\vspace*{-4mm}
    \caption{In the Monte Carlo paradigm, Brownian motion is discretized and then mapped to a numerical solution of the SDE. Despite the random fluctuations of underlying Brownian motion, which can occasionally be very large, numerical methods for SDEs typically only use fixed step sizes.}
    \label{fig:montecarlo}
\end{figure}

In this paper, we will study numerical methods with adaptive step size control for general Stratonovich SDEs of the form:
\begin{align}
    \label{eq:strat_SDE}
    dy_t = f(y_t)\m dt + g(y_t) \circ dW_t\m,
\end{align}
where $y=\{y_t\}_{t\in[0,T]}$ is an $\R^e$-valued process, $W = (W^1, \cdots, W^d) = \{W_t\}_{t\in [0,T]}$ is a $d$-dimensional Brownian motion, the vector fields $f:\R^e\rightarrow\R^e$ and $g = (g_1, \cdots, g_d):\R^e\rightarrow \R^{e\times d}$ are sufficiently smooth and $g(y_t) \circ dW_t$ is understood as $\sum_{i=1}^d g_i(y_t)\circ dW_t^i$.\medbreak

In \cite{gaines1997variable}, Gaines and Lyons consider adaptive step size controls for SDE simulation where the current ``candidate'' step size $h_k$ at time $t_k$ may be halved depending on the increment of the underlying Brownian motion $W_{t_k+h_k} - W_{t_k}$ sampled on $[t_k,t_k+h_k]$.
This approach is intuitive since it can allow the numerical method for (\ref{eq:strat_SDE}) to adapt to the ``large'' fluctuations of Brownian motion which, despite having a low probability, can significantly increase the overall approximation error. However, such adaptive step sizes are clearly \textit{non-previsible} and therefore carry theoretical and practical challenges.

\begin{definition}
    We say that an adaptive step size control is ``previsible'' if the next time $t_{k+1}$ for the numerical solution can be determined using the information available up to time $t_k\m$. Otherwise, we say the adaptive step size is ``non-previsible''.
\end{definition}
\noindent
To address these challenges, Gaines and Lyons made the following key contributions.\medbreak
\begin{itemize}[leftmargin=1.5em]
\item Using L\'{e}vy's well-known construction of Brownian motion, a tree structure called the \textit{Brownian tree} was proposed for the generation and storage of Brownian paths. The first level stores increments of Brownian motion over intervals of a fixed size, whereas subsequent levels are dynamically created whenever the Brownian motion is sampled at midpoints within the intervals. Thus, the finest discretization of the Brownian path, which is used to compute a numerical solution of the SDE (\ref{eq:strat_SDE}), is stored at the lowest level of the tree. We illustrate a Brownian tree in Figure \ref{fig:brownian_tree}.\vspace{-4mm}
\begin{figure}[!hbt]
    \hspace*{6.5mm}\includegraphics[width=0.95\textwidth]{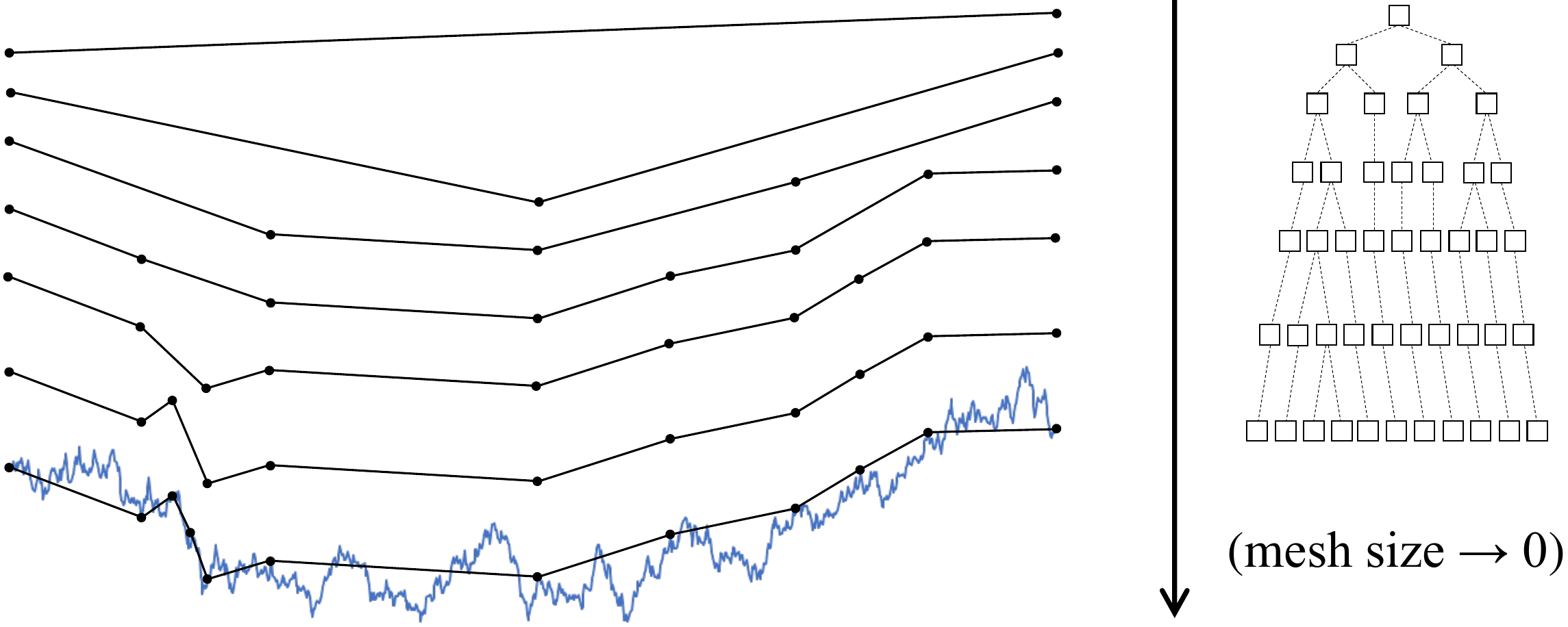}\vspace{-3mm}
    \caption{If an SDE solver has an adaptive step size control which can only halve step sizes, then it does not skip times and yields a sequence of partitions with increments stored in a Brownian tree. Of course, as the mesh of these partitions tends to zero, we would hope the approximation converges.}
    \label{fig:brownian_tree}
\end{figure}

In scientific machine learning, computing gradients of loss (or error) functions with respect to the parameters of a differential equation can be memory intensive \cite{rackauckas2021universal}. As a consequence, memory-efficient extensions of the Brownian tree have been developed, such as the \textit{Virtual Brownian Tree} \cite{jelincic2024VBT, li2020NSDEs} and \textit{Brownian Interval} \cite{kidger2021NSDEs2}.\medbreak

\item Separate from the practical challenge of implementing numerical methods for SDEs with adaptive step sizes, there is a ``simple'' theoretical question to be addressed:\smallbreak
\begin{center}
\textit{Do numerical methods for SDEs with non-previsible adaptive step sizes converge?}\medbreak
\end{center}

It was show in \cite[Theorem 4.3]{gaines1997variable} that pathwise convergence in the adaptive step size setting does still hold provided that the numerical solution $Y=\{Y_k\}$ satisfies
\begin{equation*}
Y_{k+1} = Y_k + f(Y_k)h_k + g(Y_k)W_k + \sum_{i,j=1}^d g_j^\prime(Y_k)g_i(Y_k) \int_{t_k}^{t_{k+1}} \hspace*{-1mm}\big(W_t^i - W_{t_k}^i \big)\m\circ\m dW_t^j + o(h_k),
\end{equation*}
where $W_k := W_{t_{k+1}} - W_{t_k}$ and $h_k := t_{k+1} - t_k$ denote the \textit{Brownian increment} and step size used to propagate the numerical solution $Y_k$ over the interval $[t_k\m,\m t_{k+1}]$.\medbreak

\item To motivate this pathwise convergence result, Gaines and Lyons present a simple counterexample \cite[Section 4.1]{gaines1997variable} demonstrating that the Euler-Maruyama method, which clearly does not satisfy the above condition, can converge to an incorrect (i.e.~not It\^{o}) solution of the SDE when the adaptive step sizes are non-previsible.
\end{itemize}\medbreak

Therefore, it is proposed in \cite{gaines1997variable} that second iterated integrals of Brownian motion
should be finely discretized (with $o(h_k)$ accuracy) to ensure the SDE solver converges.
As a consequence, following \cite{gaines1997variable}, research has focused on either previsible step sizes \cite{fang2020adaptiveeuler, hoel2022highorderadaptive, hofmann2000additive, kelly2018adaptive, lamba2007adaptiveuler, mora2023variable, neuenkirch2019adaptiveeuler, romisch2006adaptive}, non-previsible adaptive step sizes for commutative SDEs \cite{burrage2002variable, burrage2004variable, foster2022sle, ilie2015adaptive, jelincic2024VBT, mauthner1998variable, rackauckas2017adaptive, adaptive2021brownian} or adaptive high order schemes with iterated integrals \cite{bayer2023adaptive}.\medbreak

Unlike previous works, we consider the use of non-previsible adaptive step sizes for simulating general SDEs -- but where the solver does not require iterated integrals $\int_{t_k}^{t_{k+1}} \hspace*{-1mm}\big(W_t^i - W_{t_k}^i \big)\m\circ\m dW_t^j\m$, which are computationally expensive to finely discretize \cite{foster2024signature}.\medbreak

In this paper, we argue that such integrals are not needed and traditional methods, such as the no$\m$-area Milstein and Heun methods, can converge with adaptive step sizes.
More concretely, we shall consider a class of adaptive step sizes that produce ``no-skip'' partitions when applied to discretize an SDE. This notion is formally defined below.

\begin{definition}[No-skip partition] We say that $\D$ is a ``no-skip'' partition of $[0,T]$ with respect to a Brownian motion $W$ if there exists measurable functions $\{\Lambda_l\}_{l\m\geq\m 0}$ where $\Lambda_l : \R^{n_l\times m_l} \rightarrow [0,T]^{n_{l+1}}$ for positive integers $\{n_l\}_{l\m\geq\m 0}$ and $\{m_l\}_{l\m\geq\m 0}\m$, and an initial time $t\in(0,T]$ such that\vspace{-2.25mm}
 \begin{align*}
\D = \bigcup_{l = 0}^L \mathcal{T}_l\m,\\[-22pt]
\end{align*}
where:\smallbreak
\begin{itemize}
    \item $\mathcal{T}_0 := \{0, t\}$ and $\mathcal{T}_l = \{t_i^l\}_{0\leq i\leq n_l}$ is a collection of times $t_i^l\in[0,T]$ obtained as
\begin{align*}
\mathcal{T}_{l+1} & := \{0\}\cup\Lambda_l(\mathcal{E}_l).
\end{align*}

    \item $L$ is a stopping time with respect to the filtration $\mathcal{F}^\D$ defined by the $\sigma$-algebras
\begin{align*}
\mathcal{F}_1^\D & :=  \sigma\big(\mathcal{E}_1\big),\\[2pt]
\mathcal{F}_{l+1}^\D & := \sigma\big(\mathcal{F}_{l}^\D \cup \mathcal{E}_{l+1}\big).
\end{align*}

    \item $\mathcal{E}_l$ contains the following information corresponding to the $n_l$ timesteps in $\mathcal{T}_l\m$,
\begin{align*}
\mathcal{E}_l := \left\{\big(h_k^l\m, \W_k^{\m l}\m, z_k^l\m\big) \,:\hspace{2.5mm}\begin{matrix}h_k^l = t_{k+1}^l - t_k^l\m,\hspace{2.5mm} z_k^l\in\R^{m_l-(q+2)},\hspace{2.5mm}0\leq k < n_l\\[4pt] \hspace{-8.6mm}\W_k^{\m l} = \Big\{\int_{t_k^l}^{t_{k+1}^l} \Big(\frac{t-t_k^l}{t_{k+1}^l - t_k^l}\Big)^p\, d W_t : 0\leq p\leq q\Big\}\end{matrix}\right\},
\end{align*}
where $q\geq 0$ and $z_k^l$ is a random vector independent of the Brownian motion.
\end{itemize}
\end{definition}\medbreak

\begin{definition}\label{def:no_skip_dyadic_partition}
We say $\D$ is a no-skip dyadic partition of $[0,T]$ if it is a no-skip partition and, for each $\omega\in\Omega$, $\D(\omega) = \D_N(\omega)$ for some $N(\omega)\geq 0$ where $\D_0 := \{0,T\}$ and $\m\D_{n+1} := \{\frac{1}{2}(t_k+t_{k+1})\}\cup \D_n$ for some choice of consecutive times $t_k\m,t_{k+1}\in\D_n\m$.
\end{definition}\smallbreak

\begin{example}
The adaptive step sizes described in Gaines and Lyons \cite{gaines1997variable} give no-skip partitions. In their case, the function $\Lambda_l$ will halve a step size depending on an estimate of the local error which is determined by the Brownian increments in $\W_k^{\m l}\m$. $L$ is the stopping time of when the proposed numerical solution is accepted at time $T$.
\end{example}

\medbreak

\noindent
With these definitions, we give an informal version of our main result, Theorem \ref{thm:pathwise_conv_thm}.\smallbreak
\begin{theorem}[Convergence of adaptive SDE approximations, informal version]
Suppose that the vector fields of the SDE (\ref{eq:strat_SDE}) are bounded and differentiable with bounded derivatives. Moreover, suppose $g$ is twice differentiable with $g^{\prime\prime}$ bounded and $\alpha$-H\"{o}lder continuous where $\alpha\in(0,1)$. We assume the following about the SDE solver:\label{thm:intro}\smallbreak
\begin{enumerate}[leftmargin=1.5em]
\item Let $Y = \{Y_k\}$ denote a numerical solution of (\ref{eq:strat_SDE}) where $Y_0 := y_0$ and each step $Y_k\mapsto Y_{k+1}\m$, between times $t_k$ and $t_{k+1}\m$, is a function of the (random) quantities:
\begin{align}\label{eq:random_quantities}
h_k = t_{k+1}-t_k\m,\hspace{2mm}\W_k := \bigg\{\int_{t_k}^{t_{k+1}} \bigg(\frac{t-t_k}{t_{k+1} - t_k}\bigg)^p\, d W_t\bigg\}_{0\m\leq\m p\m\leq\m q},\hspace{2mm} z_k\in\R^m,
\end{align}
where $h_k$ is the step size used, $q\geq 0$ and the random vector $z_k$ is assumed to be independent from the Brownian motion. For example, $z_k$ could be the additional uniform random variable used within so-called ``randomized'' numerical methods
\cite{he2020randomized, hu2019randomized, kruse2021randomized, morkisz2021randomized, shen2019randomized}. In addition, we assume that each step of the method satisfies 
\begin{align}\label{eq:intro_ode_condition}
Y_{k+1} = Y_k & + f(Y_k)h_k + g(Y_k)W_k\\
& + \sum_{i,\m j\m=\m 1}^d g_j^\prime(Y_k)g_i(Y_k)\,\E\bigg[\int_{t_k}^{t_{k+1}} \hspace*{-1mm}\big(W_t^i - W_{t_k}^i \big)\circ dW_t^j\,\Big|\, \W_k\bigg]\nonumber\\[1pt]
& + R_k\big(Y_k\m, h_k\m,\W_k\m, z_k\big)\m,\nonumber
\end{align}
where the remainder $R_k$ is assumed to be small so that $R_k\sim o(h_k)$ almost surely.\medbreak

\item We assume the adaptive step size control we are using produces a sequence $\{\D_n\}$ of no-skip dyadic partitions with respect to $W$ and $\mesh(\D_n) \rightarrow 0$ almost surely.\medbreak

\end{enumerate}
Then, letting $Y^n$ denote the numerical solution computed over $\D_n = \{t_k^n\}_{k\geq 0}\m$, we have
\begin{align*}
\sup_k\big\|Y_k^n - y(t_k^n)\big\|_2 \rightarrow 0,
\end{align*}
as $n\rightarrow \infty$ almost surely, where $\|\m\cdot\m\|_2$ denotes the standard Euclidean norm.
% Moreover, the approximations $\{Y^n\}$ converges to $\{y_t\}_{t\in[0,T]}$ in a ``\textit{rough path}'' sense almost surely.
\end{theorem}\smallbreak
\begin{remark}
The condition (\ref{eq:intro_ode_condition}) is much less restrictive than \cite[Equation (4.6)]{gaines1997variable} since it does not require the second iterated integrals of the Brownian motion to be finely discretized with $o(h)$ accuracy. In practice, (\ref{eq:intro_ode_condition}) can be verified as, for example,
\begin{align}
\E\bigg[\int_{t_k}^{t_{k+1}} \hspace*{-1mm}\big(W_t^i - W_{t_k}^i \big)\circ dW_t^j\,\Big|\, W_k\bigg] & = \frac{1}{2}W_k^i\m W_k^j\m,\label{eq:expected_integral_1}\\[3pt]
\E\bigg[\int_{t_k}^{t_{k+1}} \hspace*{-1mm}\big(W_t^i - W_{t_k}^i \big)\circ dW_t^j\,\Big|\, W_k\m, H_k\bigg] & = \frac{1}{2}W_k^i\m W_k^j + H_k^i\m W_k^j - W_k^i\m H_k^j\m,\label{eq:expected_integral_2}
\end{align}
where $H_k = \int_{t_k}^{t_{k+1}}\hspace{-0.75mm} \big(\frac{1}{2} - \frac{t}{h_k}\big)\m dW_t\sim\mathcal{N}(0,\frac{1}{12}h_k\m I_d)$ is independent of $W_k$ \cite[Theorem 5.4]{foster2023levyarea}.
Intuitively, the fact that these approximations are conditional expectations enables us to employ certain martingale arguments for establishing convergence in our analysis. Otherwise, the $O(h)$ local errors could propagate linearly to give a global $O(1)$ error.
\end{remark}\medbreak

To support the above remark, we give a few examples of SDE solvers which satisfy the condition (\ref{eq:intro_ode_condition}) and have an embedded method to enable adaptive step size control. Of course, these are not extensive and we hope that our general convergence analysis will enable the development of further solvers with non-previsible adaptive step sizes.\medbreak

\subsection{Examples of adaptive numerical methods}\label{sect:examples_of_methods}

We introduce the notation
\begin{align}\label{eq:new_notation1}
F(y) := \begin{pmatrix}
f(y) &  g(y)
\end{pmatrix},\quad \overline{W}_{\n t} := \begin{pmatrix}
t & W_t
\end{pmatrix}^\T,
\end{align}
so that the SDE (\ref{eq:strat_SDE}) can be rewritten as a simpler ``controlled'' differential equation:
\begin{align}\label{eq:sde_as_cde}
dy_t = F(y_t) \circ d\overline{W}_t\m.
\end{align}
To define our examples of adaptive numerical methods, we also introduce the notation
\begin{align}\label{eq:new_notation2}
\overline{W}_{\n k} := \begin{pmatrix}
h_k & W_k
\end{pmatrix}^\T,\quad\quad\overline{H}_k := \begin{pmatrix}
0 & H_k
\end{pmatrix}^\top,
\end{align}
where $H_k := \int_{t_k}^{t_{k+1}}\hspace{-0.25mm} \big(\frac{1}{2} - \frac{t}{h_k}\big)\m dW_t = \frac{1}{h_k}\int_{t_k}^{t_{k+1}}\hspace{-0.25mm}(W_t - W_{t_k} - \frac{t-t_k}{h_k}W_k)\m dt\sim\mathcal{N}(0,\frac{1}{12}h_k\m I_d)$ is the ``space-time'' L\'{e}vy area of Brownian motion on $[t_k\m,t_{k+1}]$, see \cite[Definition 3.5]{foster2020poly}.\medbreak
% and the random vector $S_k\sim\text{Uniform}\big(\big\{-1,1\big\}^d\big)$ is assumed to be independent of $W$.

We now present a few straightforward numerical methods that are compatible with the class of adaptive step size controls satisfying the conditions of Theorem \ref{thm:intro}.\medbreak

Due to their popularity in ODE numerics, we propose embedded methods which output $Y_{k+1}$ alongside another approximation $\widetilde{Y}_{k+1}$ from $(Y_k, h_k, \W_k, z_k)$ at each step (often by reusing vector field evaluations). The difference between $Y_{k+1}$ and $\widetilde{Y}_{k+1}$ may then serve as a proxy for the local error and used by the adaptive step size controller. For example, the decision to halve a step size $h_k$ can simply be made if $\|Y_{k+1} - \widetilde{Y}_{k+1}\|_2$ is above a user-specified tolerance. In particular, this should improve accuracy when ``large'' Brownian increments $W_k$ are generated (which happens with low probability).\medbreak

Whilst \cite{gaines1997variable} showed that the Euler-Maruyama method may fail to converge to the It\^{o} SDE solution when used with an adaptive step size, our main result establishes the convergence of the following ``increment only'' Milstein method introduced in \cite{giles2014antithetic}.

\begin{definition}[No$\m$-area Milstein method with embedded Euler for It\^{o} SDEs]\label{def:milstein_euler}
We define a numerical solution $\{Y_k\}$ for the below SDE (understood in the It\^{o} sense)
\begin{align}
    \label{eq:ito_SDE}
    dy_t = f(y_t)\m dt + g(y_t)\m dW_t\m,
\end{align}
by $Y_0 := y_0$ and, for $k\geq 0$, we define $Y_{k+1}$ and the embedded approximation $\widetilde{Y}_{k+1}$ as
\begin{align}
Y_{k+1} & := Y_k + f\big(Y_k\big) h_k + g\big(Y_k\big) W_k + \frac{1}{2}g^{\m\prime}\big(Y_k\big)g\big(Y_k\big)\big(W_k^{\otimes 2} - h_k I_d\big),\label{eq:milstein}\\
\widetilde{Y}_{k+1} & := Y_k + f\big(Y_k\big) h_k + g\big(Y_k\big) W_k\m.\label{eq:milstein_embed}
\end{align}
\end{definition}
\begin{remark}\label{rmk:milstein}
To apply our main result, we rewrite (\ref{eq:ito_SDE}) as a Stratonovich SDE:
\begin{align*}
    dy_t = \widetilde{f}(y_t)\m dt + g(y_t)\circ dW_t\m,
\end{align*}
where $\widetilde{f}(y) := f(y) - \frac{1}{2}\sum_{i=1}^d g_i^\prime(y)g_i(y)$ is defined using the It\^{o}$\m$-Stratonovich correction. Then the no$\m$-area Milstein method given by (\ref{eq:milstein}) simplifies and can be expressed as
\begin{align*}
Y_{k+1} := Y_k + \widetilde{f}\big(Y_k\big) h_k + g\big(Y_k\big) W_k + \frac{1}{2}g^{\m\prime}\big(Y_k\big)g\big(Y_k\big)W_k^{\otimes 2}.
\end{align*}
By equations (\ref{eq:intro_ode_condition}) and (\ref{eq:expected_integral_1}), with $R_k = 0$, it is clear the above has the required form.
\end{remark}

To our knowledge, the following are the first embedded stochastic Runge$\m$-Kutta methods for general SDEs (see \cite{embed2003srk, adaptive2021brownian} for examples with scalar or additive noise). We first consider a multidimensional extension of the Heun scheme proposed in \cite{roberts2012euler}.

\begin{definition}[Heun's method with embedded Euler-Maruyama]\label{def:heun_euler} Using the notation in (\ref{eq:new_notation1}) and (\ref{eq:new_notation2}), we define a numerical solution $\{Y_k\}$ for the SDE (\ref{eq:sde_as_cde}) as
\begin{align}
\widetilde{Y}_{k+1} & := Y_k + F\big(Y_k\big)\overline{W}_{\n k}\m,\label{eq:heun_embed}\\[2pt]
Y_{k+1} & := Y_k + \frac{1}{2}\big(F\big(Y_k\big) + F\big(\widetilde{Y}_{k+1}\big)\big)\overline{W}_{\n k}\m,\label{eq:heun}
\end{align}
for $k\geq 0$ with $Y_0 := y_0\m$.
\end{definition}

In addition, we shall introduce a new stochastic Runge$\m$-Kutta method that uses space-time L\'{e}vy area $H_k$ to achieve a slightly better local accuracy than traditional ``increment only'' schemes, including Heun's method (see Appendix \ref{append:taylor_expansions} for details).
For additive noise SDEs, the proposed scheme becomes high order and achieves an $O(h^{1.5})$ strong convergence rate (i.e.~when the noise vector fields $g_i$ are constant). This scheme was inspired by a recent path-based splitting method \cite[Example 1.3]{foster2024splitting}.

\begin{definition}[Splitting Path Runge-Kutta (SPaRK) with embedded Heun] We define a numerical solution $\{Y_k\}$ for the SDE (\ref{eq:sde_as_cde}) by $Y_0 := y_0$ and, for all $k\geq 0$,\label{def:srk_heun}\vspace*{-1.5mm}
\begin{align}
Y_{k+\frac{1}{2}} := Y_k & + F\big(Y_k\big)\bigg(\frac{1}{2}\overline{W}_{\n k} + \sqrt{3}\,\overline{H}_k\bigg),\hspace{7.5mm}
Z_{k+1} := Y_k +  F\big(Y_{k+\frac{1}{2}}\big)\overline{W}_{\n k}\m,\nonumber\\
Y_{k+1} := Y_k & + F\big(Y_k\big)\big(a\m\overline{W}_{\n k} + \overline{H}_k\big) + b\m F\big(Y_{k+\frac{1}{2}}\big)\overline{W}_{\n k} + F\big(Z_{k+1}\big)\big(a\m\overline{W}_{\n k} - \overline{H}_k\big),\label{eq:srk}\\
\widetilde{Y}_{k+1} := Y_k &  + \frac{1}{2}\big(F\big(Y_k\big) + F\big(Z_{k+1}\big)\big)\overline{W}_{\n k}\m,\label{eq:srk_embed}\\[-20pt]\nonumber
\end{align}
where $a := \frac{3-\sqrt{3}}{6}$ and $b := \frac{\sqrt{3}}{3}$.
\end{definition}

Further details on the SPaRK method (\ref{eq:srk}) and its improved local accuracy over Heun's method (particularly for SDEs with general noise) are given in Appendix \ref{append:taylor_expansions}. However, the formulation of the above (embedded) Runge$\m$-Kutta methods is only one component of the SDE solver -- the other being an algorithm for controlling step sizes.\medbreak

For example, we could consider the following simple adaptive step size control, though we emphasise that our analysis (Theorem \ref{thm:pathwise_conv_thm}) holds more generally, provided that their associated partitions are no-skip and dyadic with mesh size tending to zero.

\begin{definition}[A simple dyadic step size control for embedded Runge-Kutta]Let $h_\text{init}\,, C > 0$ be fixed user-specified constants (corresponding to the initial step size and local error tolerance). Then we  depth-first recursively construct a Brownian tree, with the $n$-th level containing Brownian information on intervals of size $h_\text{init}\m / 2^n\m$, until\label{def:intro_adaptive_control}
\begin{align}\label{eq:intro_adaptive_control}
\big\|Y_{k+1} - \widetilde{Y}_{k+1}\big\|_2 \leq C \sqrt{h_k}\m,
\end{align}
holds for all $k\geq 0$, where the numerical solution $Y = \{Y_k\}_{k\m\geq\m 0}$ is constructed using the information of the Brownian motion provided by the leaves of the Brownian tree.
\end{definition}
\begin{remark}
In SDE numerics, it is common to consider mean squared errors. Informally speaking, in the constant step size regime, a local mean squared error of $O(h^p)$ gives a global mean squared error of $O(h^{p-1})$ (see \cite[Theorem 1.1]{milstein2004physics} for details).
Therefore, if we would like the global error to be less than $\varepsilon > 0$, then we would expect local mean squared errors to be at most $O(\varepsilon h)$. This motivates the estimator (\ref{eq:intro_adaptive_control}), although $\|Y_{k+1} - \widetilde{Y}_{k+1}\|_2^2$ is a proxy for the pathwise (and not the mean squared) error.
\end{remark}\medbreak
\begin{remark}
Suppose there exists an almost surely finite random variable $c$ such that $\|Y_{k+1} - \widetilde{Y}_{k+1}\|_2\geq c h_k^\beta$ with $\beta > \frac{1}{2}$. Then, by decreasing the tolerance $C\rightarrow 0$, we have  $\sup_k(h_k)\rightarrow 0$ as $\m c h_k^\beta \leq \|Y_{k+1} - \widetilde{Y}_{k+1}\|_2\leq C \sqrt{h_k}\m$ implies that $h_k^{\beta - \frac{1}{2}} \leq \frac{C}{c}$.
However, we have been unable to show this lower bound, and thus can only conjecture that the step size controller in Definition \ref{def:intro_adaptive_control} satisfies the assumptions of our result.\label{rmk:intro_adaptive_control}
\end{remark}
\subsection{Organisation of the paper}\label{sect:organisation} In Section \ref{sect:main_result}, we will establish our pathwise convergence result (Theorem \ref{thm:pathwise_conv_thm}) for a wide class of adaptive SDEs approximations. The key idea in our proof methodology is to employ \textit{rough path theory} and show that piecewise polynomial approximations of Brownian motion converge in the \textit{$\alpha\m$-H\"{o}lder} metric along any nested subsequence of the partitions produced by the SDE solver. This argument is a natural extension of the analysis presented in \cite[Section 13.3.2]{frizvictoir2010roughpaths}, which presents similar results for piecewise linear approximations of Brownian motion that are defined along a sequence of partitions that are nested, but also deterministic.\medbreak

Once we have $\alpha\m$-H\"{o}lder convergence for the piecewise polynomial approximations, we can apply an extension of the \textit{Universal Limit Theorem} from rough path theory \cite{friz2024rsdes} to show that the corresponding sequence of Controlled Differential Equations (CDEs) converges to the SDE solution almost surely. We show that pathwise convergence is still preserved when small $o(h)$ local perturbations are added to the CDE solutions. By considering the local Taylor expansion of the CDEs between discretization times, we obtain the condition (\ref{eq:intro_ode_condition}) from Theorem \ref{thm:intro} for the numerical method to satisfy.
We conclude Section \ref{sect:main_result} by showing the no$\m$-area Milstein, Heun and Splitting Path Runge-Kutta (SPaRK) methods in Section \ref{sect:examples_of_methods} satisfy the convergence condition (\ref{eq:intro_ode_condition}).\medbreak

In Section \ref{sect:counterexample}, we give an example demonstrating that without no-skip partitions, an adaptive step size control can induce a non-zero bias which prevents convergence.
Here, the decision to reduce a step size is based on whether it increases the value of the numerical solution. Hence, these step sizes can ``skip over'' known values of $W$. Ultimately, this leads to an approximation with an incorrect mean -- even in the limit.\medbreak

In Section \ref{sect:numerical_example}, we present an experiment comparing methods with different step size controls for the well-known SABR\footnote{\textbf{S}tochstic \textbf{A}lpha-\textbf{B}eta-\textbf{R}ho} model used in mathematical finance \cite{SABR2017Cai, SABR2018Cui, SABR2002Hagan, SABR2017Leitao}.
The SABR model describes the evolution of an interest (or exchange) rate $S = \{S_t\}$ whose local volatility $\sigma = \{\sigma_t\}$ is also stochastic. It is given by the following It\^{o} SDE:
\begin{align}\label{eq:intro_SABR}
\begin{split}
dS_t & = \sqrt{1-\rho^2}\m \sigma_t (S_t)^\beta  dW_t^1 + \rho\m \sigma_t (S_t)^\beta  dW_t^2,\\[3pt]
d\sigma_t & = \alpha\m\sigma_t\m dW_t^2,
\end{split}
\end{align}
where $W$ is a standard two-dimensional Brownian motion and $\alpha,\beta,\rho$ are parameters.
For simplicity, we will set $\alpha = 1$ and $\beta = \rho = 0$. In Stratonovich form, (\ref{eq:intro_SABR}) becomes
\begin{align}\label{eq:intro_simple_SABR}
\begin{split}
dS_t & = \sigma_t\circ dW_t^1,\\
d\sigma_t & = -\frac{1}{2}\sigma_t\m dt + \sigma_t\circ dW_t^2.
\end{split}
\end{align}
Despite its simple form, the SDE (\ref{eq:intro_simple_SABR}) is challenging to accurately discretize since its Taylor expansion contains the non-Gaussian stochastic integral of $W^2$ against $W^1$,
\begin{align}\label{eq:intro_SABR_error}
S_{t+\Delta t} = S_t & + \sigma_t \big(W_{t+\Delta t}^1 - W_t^1\big) + \sigma_t \int_t^{t+\Delta t} \big(W_u^2 - W_t^2\big)  \circ dW_u^1  + R_{\Delta t}(t),
\end{align}
where the remainder is $\m R_{\Delta t}(t) = \sigma_t\int_t^{t+\Delta t} \big(e^{-\frac{1}{2}(u-t) + W_u^2 - W_t^2} - \big(1 + W_u^2 - W_t^2\big)\big)\circ dW_u^1\m$.
Therefore by (\ref{eq:intro_SABR_error}), approximating $S_{t+\Delta t}$ using one step of Heun's method, with a sufficiently small step size, will result in a mean squared error of $\frac{1}{4}\sigma_t^2 (\Delta t)^2 + O((\Delta t)^3)$.
As a consequence, we would expect that the overall accuracy of a numerical method can be improved by decreasing $\Delta t$ if $\sigma_t$ when large and increasing $\Delta t$ when $\sigma_t$ is small.\medbreak

We observe that solvers which use previsible steps of the form $\Delta t = \log(1+C \sigma_t^{-2})$ achieve an order of magnitude more accuracy than they do with constant step sizes. 
However, in this experiment, we are comparing the mean squared error of the SDE solvers -- which is very different from the pathwise convergence studied in the paper.\medbreak

In addition, we also see that similar improved convergence rates can be obtained if local errors are estimated via an embedded Runge$\m$-Kutta method and subsequently used by a ``Proportional-Integral'' (PI) adaptive step size controller based on \cite{burrage2004variable, ilie2015adaptive}. However, due to our counterxample, we modify the PI step size controller so that it does not ``skip over'' times where Brownian information has previously been sampled.
Based on the results of the experiment, we believe that our modified PI adaptive step size controller could be a promising technique for simulating non-commutative SDEs.
\medbreak

In Section \ref{sect:conclusion}, we will conclude the paper and briefly discuss potential future topics.

\subsection{Notation} In this subsection, we summarise the notation used in the paper.
We use $\|\cdot\|_2$ to denote the usual Euclidean norm. For a sequence of random variables $\{X_n\}$ in a probability space $(\Omega, \mathcal{F}, \mathbb{P})$, we say $X_n\rightarrow X$ in $L^2(\P)\m$ if $\,\E\big[\|X_n - X\|_2^2\big] \rightarrow 0$.
For $w : \R_+\rightarrow \R_+$, we say that $w(h)\sim o(h)$ if $\frac{w(h)}{h}\rightarrow 0$ as $h\rightarrow 0$ and $w(h)\sim O(h^p)$ with $p > 0$, if there exists a constant $C$ such that $w(h) \leq Ch^p$ for sufficiently small $h$.\medbreak

Throughout, $W = \{W_t\}_{t\in[0,T]}$ will be a standard $d$-dimensional Brownian motion. For a sequence of (random) no-skip partitions $\{D_n\}$ of $[0,T]$, we use $\widetilde{W}^n = \{\widetilde{W}_t^n\}_{t\in[0,T]}$ to denote a piecewise polynomial approximation of $W$ defined with respect to $\D_n\m$.
Most notably, this can be the usual piecewise linear discretization (see \cite{foster2020poly} for details). We write increments of the paths as $W_{s,t} := W_t - W_s$ and $\widetilde{W}_{s,t}^n := \widetilde{W}_t^n - \widetilde{W}_s^n$ for $s\leq t$.\medbreak

We will denote SDE solutions by $y = \{y_t\}_{t\in[0,T]}$ and CDE solutions by $\widetilde{x}^{\m n}$ or $\widetilde{y}^{\m n}$.
Each partition can be expressed as $\D_n = \{0 = t_0^n < t_1^n < \cdots < t_{K_n - 1}^n < t_{K_n}^n = T\}$ or $\{\D_n\} = \{h_k^n\}_{0\leq k < K_n}$ where $K_n\geq 1$ and $h_k^n := t_{k+1}^n - t_k^n$ represents the step size.
We often use the shorthand notation $W_k := W_{t_{k+1}} - W_{t_k}$ when defining numerical methods -- which are denoted by $Y = \{Y_k\}_{k\geq 0}$. Similarly, we write $H_k\sim\mathcal{N}(0,\frac{1}{12}h_k\m I_d)$ for the \textit{space-time L\'{e}vy area} of Brownian motion on $[t_k\m,t_{k+1}]$, see \cite[Definition 3.5]{foster2020poly}.\medbreak

Given a random vector $X$, we let $\sigma(X)$ denote the $\sigma$-algebra generated by $X$. Filtrations are denoted by $\{\F_n\}$ or $\{\G_n\}$ and give information about $W$ and partitions.
Blackboard bold symbols lie in $\R^{d\times d}$, such as iterated integrals of Brownian motion $\mathbb{W}_{s,t} = \big\{\int_s^t W_{s,u}^i \circ dW_u^j\big\}_{1\leq i,j\leq d}$. Bold symbols lie in the tensor algebra $\R\oplus \R^d\oplus \R^{d\times d}$, such as the \textit{Stratonovich enhanced} Brownian motion $\boldsymbol{W} = \{\boldsymbol{W}_{\hspace{-0.75mm}s,t}\} = \{(1, W_{s,t}, \mathbb{W}_{s,t})\}$.\medbreak

We use $\otimes$ to denote the usual tensor product with $a\otimes b := \{a_i\m b_j\}_{1\leq i, j\leq d} \in\R^{d\times d}$ for vectors $a = \{a_i\}_{i=1}^d$ and $b = \{b_i\}_{i=1}^d$ in $\R^d$. In particular, we can use this notation to define the iterated integrals $\m\mathbb{W}_{s,t} := \int_s^t W_{s,u} \otimes \circ\, dW_u\m$ and $\m\mathbb{\widetilde{W}}_{s,t}^n := \int_s^t \widetilde{W}_{s,u}^n \otimes d\widetilde{W}_u^n\m$.
There will also be notation for rough paths $\boldsymbol{X}$ and $\boldsymbol{Y}$ (see Definition \ref{def:rough_path_stuff}), namely the $\alpha$-H\"{o}lder norm $\|\boldsymbol{X}\|_{\alpha\text{-H\"{o}l;[s,t]}}\m$ and $\alpha$-H\"{o}lder metric $d_{\alpha\text{-H\"{o}l;[s,t]}}\big(\boldsymbol{X}, \boldsymbol{Y}\big)$ where $\alpha \in \big(\frac{1}{3}, \frac{1}{2}\big)$.
For a vector $\boldsymbol{X} = (1, X, \mathbb{X})\in \R \oplus \R^d\oplus \R^{d\times d}$, we also define $\|\boldsymbol{X}\| := \max\big(\|X\|_2, \|\mathbb{X}\|_2^\frac{1}{2}\big)$.
\smallbreak

\section{Main result}\label{sect:main_result}

We will first show the ``rough path'' convergence of piecewise polynomial approximations, defined on no-skip dyadic partitions, to Brownian motion.
We note that, unlike \cite[Section 13.3.2]{frizvictoir2010roughpaths}, which our analysis is based on, the sequence of nested partitions that we consider is allowed to depend on the Brownian motion.
However, due to their no-skip property, our nested partitions will still produce a filtration -- which will facilitate the martingale arguments we use to show convergence.
We note that this strategy for showing convergence was originally proposed in the doctoral thesis \cite[Chapter 6]{foster2020thesis}. However, in this paper, we introduce the notion of no-skip dyadic partitions and consider a more general class of adaptive step sizes which can depend on random variables that are independent of the Brownian motion.\medbreak

Before starting our analysis, we will make the following remark and definitions:

\begin{remark}[Writing a partition as a collection of step sizes]\label{def:paritions_as_steps} Given a partition $\D = \{0 = t_0 < \cdots < t_N = T\}$ of $[0,T]$, we can equivalently express $\D$ in terms of its step sizes as $\D = \{h_k\}_{0\leq k < N}$ where $h_k := t_{k+1} - t_k\m$.
\end{remark}\smallbreak

\begin{definition}[Enhancement function] For a $d$-dimensional Brownian motion $W$ and a no-skip partition $\D = \{h_k\}$ of $\m[0,T]$, we will define the following functions,\label{def:enhancement} 
\begin{align}
\mathcal{E}_{\W}\big(\{h_k\}, W\big) & := \bigg\{\big(h_k, \W_k\big) : \W_k = \bigg\{\int_{t_k}^{t_{k+1}} \Big(\frac{t-t_k}{t_{k+1} - t_k}\Big)^p\, d W_t\bigg\}_{0\m\leq\m p\m\leq\m q}\bigg\}\m,\label{eq:path_enhancement}\\[3pt]
\mathcal{E}_{\mathcal{Z}}\big(\{h_k\}, W\big) & := \Big\{\big(h_k, z_k\big) : z_k \in\R^m\Big\},\label{eq:area_enhancement}
\end{align}
where $q\in\{0,1,\cdots\,\}$ and $z_k$ is the random vector in $\R^m$ (with $m\geq 1$) independent of $W$. In particular, we can view $z_k$ as corresponding to all of the random vectors that were independent of $W$ in the definition of no-skip partition (Definition \ref{def:no_skip_dyadic_partition}). Combining functions (\ref{eq:path_enhancement}) and (\ref{eq:area_enhancement}), we define the ``enhancement'' function $\mathcal{E}$ as
\begin{align}\label{eq:enhancement}
\mathcal{E}\big(\{h_k\}, W\big) := \mathcal{E}_{\W}\big(\{h_k\}, W\big)\cup \mathcal{E}_{\mathcal{Z}}\big(\{h_k\}, W\big).
\end{align}
\end{definition}\medbreak
\begin{definition}[$W$-adapted partitions]
We say that a sequence of partitions $\{\D_n\}$ is adapted to a Brownian motion $W$ if the following $\sigma$-algebras give a filtration,\label{def:w_adapted}
\begin{align}\label{eq:filtration}
\F_n := \sigma\big(\mathcal{E}_{\W}\big(\D_n\m, W\big)\cup\G_n\big),
\end{align}
where the filtration $\{\G_n\}$ is defined by $\hspace{0.4mm}\G_1 := \sigma\big(\mathcal{E}_{\mathcal{Z}}\big(\D_1\m, W\big)\big)\hspace{-0.25mm}$ and for $n\geq 1$,
\begin{align}\label{eq:area_filtration}
\G_{n+1} := \sigma\big(\G_n\cup \mathcal{E}_{\mathcal{Z}}\big(\D_{n+1}\m, W\big)\big).
\end{align}
\end{definition}

Ultimately, it will be the $W$-adapted property that we require our sequence of partitions to satisfy in order to perform our analysis. Whilst the original sequence coming from the adaptive step sizes is not nested (and thus may not be $W$-adapted), it will be possible to find a suitable subsequence. We formalise this in Theorem \ref{thm:subsequence}.

\begin{theorem}[Subsequences of no-skip dyadic partitions can be $W$-adapted]
Let $\{\D_n\}_{n\geq 1}$ denote a sequence of no-skip dyadic partitions of $[0,T]$ defined with respect to a Brownian motion $W$. Then, almost surely, for any subsequence $\{\D_n^{\m\prime}\}$ there exists a further subsequence $\{\D_n^{\m\prime\prime}\}$ which is nested and $W$-adapted.\label{thm:subsequence}
\end{theorem}
\begin{proof}
Let $\omega\in\Omega$. Since $\D_n^{\m\prime}(\omega)$ is a dyadic partition, it contains a smallest subinterval of length $2^{-k_n(\omega)}T$ for some integer $k_n(\omega) \geq 0$. Since $\mesh(\D_n)\rightarrow 0$ almost surely, we have  $\mesh(\D_n^{\m\prime}(\omega))\rightarrow 0$ for almost all $\omega\in\Omega$. In particular, there will be a subsequent partition $\D_{n^\prime}^{\m\prime}(\omega)$ where its times are at most $2^{-k_n(\omega)}T$ apart. From the dyadic assumption, it immediately follows that $\D_n^{\m\prime}(\omega)$  is a subset of $\D_{n^\prime}^{\m\prime}(\omega)$. \medbreak

Hence, for $\omega\in\Omega$, we can define a further subsequence $\D_n^{\m\prime\prime}(\omega)$ of  $\D_n^{\m\prime}(\omega)$ as follows:
\begin{align*}
\D_1^{\m\prime\prime}(\omega) & := \D_1^{\m\prime}(\omega)\m,\\[3pt]
\D_{n+1}^{\m\prime\prime}(\omega) & := \D_{\varphi(n)}^{\m\prime}(\omega)\m,\,\,\text{ where }\,\, \varphi(n) := \inf\{m > n :  \D_n^{\prime\prime}(\omega)\subseteq \D_m^{\m\prime}(\omega)\}.
\end{align*}
Thus, all that remains is to show that $\{\D_n^{\m\prime\prime}\}$ is $W$-adapted (given by Definition \ref{def:w_adapted}).
To this end, we will first show that $\sigma\big(\mathcal{E}_{\W}\big(\D_n^{\m\prime\prime}, W\big)\big) \subseteq \sigma\big(\mathcal{E}_{\W}\big(\D_{n+1}^{\m\prime\prime}, W\big) \cup \mathcal{E}_{\mathcal{Z}}\big(\D_n^{\m\prime\prime}, W\big)\big)$.\medbreak

By construction, the subsequence $\{\D_n^{\m\prime\prime}\}$ is nested, i.e.~$\D_n^{\m\prime\prime}\subseteq\D_{n+1}^{\m\prime\prime}\m$. In particular, by merging a finite number of subintervals in $\D_{n+1}^{\m\prime\prime}\m$, it is possible to obtain $\D_n^{\m\prime\prime}\m$.
Suppose we merge the subintervals $[s,u]$ and $[u,t$]. Then, for $p\in\{0,\cdots,q\}$, we have
\begin{align*}
\int_{s}^{t} \bigg(\frac{r-s}{t-s}\bigg)^p\, d W_r & = \bigg(\frac{u-s}{t-s}\bigg)^p\int_{s}^{u} \bigg(\frac{r-s}{u-s}\bigg)^p\, d W_r + \bigg(\frac{t-u}{t-s}\bigg)^p\int_{u}^{t} \bigg(\frac{r-s}{t-u}\bigg)^p\, d W_r\\[1pt]
& = \bigg(\frac{u-s}{t-s}\bigg)^p\int_{s}^{u} \bigg(\frac{r-s}{u-s}\bigg)^p\, d W_r\\
&\hspace{7.5mm} + \bigg(\frac{t-u}{t-s}\bigg)^p \sum_{l=0}^p {p\choose l}\bigg(\frac{u-s}{t-u}\bigg)^{p-l}\int_{u}^{t} \bigg(\frac{r-u}{t-u}\bigg)^l\, d W_r\m.\\[-20pt]
\end{align*}
In other words, the integrals
$\big\{\int_s^t \big(\frac{r-s}{t - s}\big)^p\m d W_r : 0 \leq p\leq q\big\}$ can be recovered from $\big\{\int_s^u \big(\frac{r-s}{u - s}\big)^p\m d W_r\big\}$ and $\big\{\int_u^t \big(\frac{r-u}{t - u}\big)^p\m d W_r\big\}$. Moreover, since $\D_n^{\m\prime\prime}$ is a no-skip partition, it can be constructed using only the information that is available in $\sigma\big(\mathcal{E}\big(\D_n^{\m\prime\prime}, W\big)\big) = \sigma\big(\mathcal{E}_{\W}\big(\D_n^{\m\prime\prime}, W\big)\cup \mathcal{E}_{\mathcal{Z}}\big(\D_n^{\m\prime\prime}, W\big)\big)$ where $\mathcal{E}_{\mathcal{Z}}\big(\D_n^{\m\prime\prime}, W\big)$ contains the extra random variables in the construction of $\D_n^{\m\prime\prime}$ that are independent of $W$. Therefore, since $\D_{n+1}^{\m\prime\prime}$ is a refinement of $\D_n^{\m\prime\prime}\m$, it follows that we can recover the information in $\sigma\big(\mathcal{E}_{\W}\big(\D_n^{\m\prime\prime}\m, W\big)\big)$ from $\sigma\big(\mathcal{E}_{\W}\big(\D_{n+1}^{\m\prime\prime}\m, W\big)\cup \mathcal{E}_{\mathcal{Z}}\big(\D_n^{\m\prime\prime}, W\big)\big)$.
Hence, we obtain the following relationship: 
\begin{align}\label{eq:filtration_working}
\sigma\big(\mathcal{E}_{\W}\big(D_n^{\m\prime\prime}\m, W\big)\big)\subseteq\sigma\big(\mathcal{E}_{\W}\big(D_{n+1}^{\m\prime\prime}\m, W\big)\cup \mathcal{E}_{\mathcal{Z}}\big(\D_n^{\m\prime\prime}, W\big)\big).
\end{align}
From the definition of $W$-adapted, we want to show that the following is a filtration:
\begin{align*}
\F_n := \sigma\big(\mathcal{E}_{\W}\big(\D_n^{\m\prime\prime}\m, W\big)\cup\G_n\big),
\end{align*}
where the $\sigma$-algebra $\G_n$ is defined as $\G_{n+1} := \sigma\big(\G_n\cup \mathcal{E}_{\mathcal{Z}}\big(\D_{n+1}^{\m\prime\prime}\m, W\big)\big)$ for $n\geq 1$ with $\G_1 := \sigma\big(\mathcal{E}_{\mathcal{Z}}\big(\D_1^{\m\prime\prime}\m, W\big)\big)$.
By ``expanding'' $\F_{n+1}$ and using the property (\ref{eq:filtration_working}), we have
\begin{align*}
\F_{n+1} & = \sigma\big(\mathcal{E}_{\W}\big(\D_{n+1}^{\m\prime\prime}\m, W\big)\cup\G_{n+1}\big),\\
& = \sigma\big(\mathcal{E}_{\W}\big(\D_{n+1}^{\m\prime\prime}\m, W\big)\cup\G_n \cup \mathcal{E}_{\mathcal{Z}}\big(\D_{n+1}^{\m\prime\prime}\m, W\big)\big)\\
& = \sigma\big(\big(\mathcal{E}_{\W}\big(\D_n^{\m\prime\prime}\m, W\big)\cup\G_n\big)\cup\big(\mathcal{E}_{\W}\big(\D_{n+1}^{\m\prime\prime}\m, W\big)\cup \mathcal{E}_{\mathcal{Z}}\big(\D_{n+1}^{\m\prime\prime}\m, W\big)\big)\big)\\
 & = \sigma\big(\F_n\cup \mathcal{E}\big(\D_{n+1}^{\m\prime\prime}\m, W\big)\big).
\end{align*}
Thus, we see that $\{\F_n\}$ is an increasing sequence of $\sigma$-algebras (i.e.~a filtration).
\end{proof}

Since we shall be using martingale convergence arguments, we will show that the filtration $\{\F_n\}$ does contain enough information to determine the Brownian motion.\smallbreak

\begin{proposition}
Let $\{\D_n^{\m\prime\prime}\}$ be the subsequence of partitions and let $\{\F_n\m, \G_n\}$ be the filtrations given by Theorem \ref{thm:subsequence} and its proof. Then $\F_\infty := \bigcup_{n\geq 0}\F_n$ is equal to
\begin{align}\label{eq:limit_filtration}
\F_\infty = \sigma\Big(\big\{W_t : t\in[0,T]\big\}\cup \G_\infty\Big),
\end{align} 
where $\G_\infty := \bigcup_{n\geq 0}\G_n\m$.\label{prop:filtration}
\end{proposition}
\begin{proof}
Since the enhancement function produces $\W_k = \big\{\int_{t_k}^{t_{k+1}} \big(\frac{t-t_k}{t_{k+1} - t_k}\big)^p\, d W_t\big\}$, we see that $\W_k$ contains the Brownian motion's increment $\int_{t_k}^{t_{k+1}} 1\, d W_t = W_{t_{k+1}}- W_{t_k}$.\smallbreak

For any fixed $t\in[0,T]$, we can find the largest time point $t_k\in \D_n^{\m\prime\prime}$ such that $t_k \leq t$.
As $0\leq t-t_k < \mesh(\D_n^{\m\prime\prime})$, it follows from the assumption that $t_k\rightarrow t$ almost surely.
Moreover, since Brownian motion is continuous almost surely, we have $W_{t_k}\rightarrow W_t$ almost surely and we can thus obtain $W_t$ from $\F_\infty$. That is, $\sigma\big(W_t : t\in[0,T]\big)\subseteq \F_\infty$. It is also clear from the definition of the filtration $\F_n$ that $\G_n\subseteq \F_n$ and thus $\G_\infty\subseteq\F_\infty$.
\smallbreak

On the other hand, as $W$ has finite $\alpha\m$-H\"{o}lder norm almost surely, with $\alpha \in (0, \frac{1}{2})$, it follows that the Riemann-Stieltjes integrals $\big\{\int_{t_k}^{t_{k+1}} \big(\frac{t-t_k}{t_{k+1} - t_k}\big)^p\m d W_t\big\}$ exist and are uniquely determined by $W$. Thus, we have $\F_n\subseteq\sigma\big(\{W_t\}_{t\in[0,T]}\cup \G_\infty\big)$ as required.
\end{proof}

The filtration $\F_n$ will be used in our analysis. However, as it only corresponds to a subsequence, we will use the following proposition to show that the original sequence $\{\D_n\}$ of no-skip dyadic partitions gives a convergent sequence of SDE approximations.\smallbreak

\begin{proposition}
Let $\{\D_n\}$ denote a sequence of no-skip dyadic partitions with $\mesh(\D_n)\rightarrow 0$ almost surely. Let $\{\D_n^{\m\prime\prime}\}$ be the subsequence given by Theorem \ref{thm:subsequence} and let $Y^n$ be the approximation of the SDE (\ref{eq:strat_SDE}) obtained using the partition $\D_n\m$.
Suppose that the approximations corresponding to $\{\D_n^{\m\prime\prime}\}$ converge to the SDE solution $y$ with respect to a metric $d(\m\cdot, \cdot)$. Then $\{Y^n\}$ converges, i.e.~$d(Y^n, y)\rightarrow 0$ as $n\rightarrow\infty$.\label{prop:subsequence}
\end{proposition}
\begin{proof}
Assume for a contradiction that $\{Y^n\}$ does not converge to $y$. Then, there exists $\varepsilon > 0$, for which we can find a subsequence $\{\widetilde{Y}^n\}$ with $d(\widetilde{Y}^n, y) \geq \varepsilon$ for all $n$. In particular, this means that all subsequences of $\{\widetilde{Y}^n\}$ will not converge. By letting $\D_n^{\m\prime}$ be the partition which was used by $\widetilde{Y}^n$, we obtain the desired contradiction.
\end{proof}\medbreak

Therefore, by the above proposition, it is enough to show convergence of only the SDE approximations obtained along the nested subsequence of partitions $\{\D_n^{\m\prime\prime}\}$.
Hence, we can use the below remark -- which we have boxed to highlight its importance.
\vspace{-7.5mm}

\begin{remark}
\begin{tabular}{|c|}
 \multicolumn{1}{c}{} \\[1pt]
 \hline\\[-9pt]
 In our analysis, we will assume without loss of generality\\[2pt]
 (by Proposition \ref{prop:subsequence}) that $\{\D_n\}$ is nested and $W$-adapted.\\[2pt]
 \hline
\end{tabular}
\end{remark}\medbreak

Having now established the filtration $\F = \{\F_n\}_{n\m\geq\m 1}$ and shown Proposition \ref{prop:filtration},
we will proceed to show that the corresponding sequence of (adaptive) approximations $\{\widetilde{W}^n\}_{n\m\geq\m 0}$ converges to Brownian motion on the interval $[0,T]$ in a \textit{rough path} sense.

\subsection{Rough path convergence of piecewise polynomial approximations} In order to show our approximations of Brownian motion converge in a rough path sense (i.e.~in an $\alpha\m$-H\"{o}lder metric), we first perform a standard mean squared analysis.\smallbreak
 
\begin{theorem}[Nested piecewise unbiased approximations of Brownian motion]
Let $\{\D_n\}$ denote a sequence of nested and $W$-adapted partitions with corresponding filtration $\{\F_n\}$ given by (\ref{eq:filtration}). In addition, suppose that $\m\mesh(\D_n) \rightarrow 0$ almost surely.
We define a sequence of approximations $\{\widetilde{W}^n\}$ for the Brownian motion on $[0,T]$ as\label{thm:basic_approximation}
\begin{align}\label{eq:brownian_approx}
\widetilde{W}_t^n := \E\big[W_t \,|\,\F_n\big]\m.
\end{align}
Then for $t\in[0,T]$, $\widetilde{W}_t^n\rightarrow W_t$ almost surely and
\begin{align}\label{eq:mean_squared_conv}
\E\Big[\big(\widetilde{W}_t^n - W_t\big)^2\m\Big]\rightarrow 0.
\end{align}
\end{theorem}
\begin{proof}
As $\F_n$ is a filtration, it follows that $\big\{\widetilde{W}^n\big\}  = \big\{\E[W \,|\,\F_n]\big\}$ is a martingale.
Moreover, each $\widetilde{W}^n$ is square-integrable as by Jensen's inequality and the tower law,
\begin{align*}
\E\Big[\big(\widetilde{W}_t^n\big)^2\Big]  = \E\Big[\E\big[W_t \,|\,\F_n\big]^2\Big] \leq \E\Big[\E\big[W_t^2 \,|\,\F_n\big]\Big] = \E\big[W_t^2\big] = t.
\end{align*}
Therefore, by Doob's martingale convergence theorem, we have that for each $t\in[0,T]$,
\begin{align}\label{eq:doob_convergence}
\E\big[W_t \,|\,\F_n\big] \rightarrow \E\big[W_t \,|\,\F_\infty\big]\m,
\end{align}
almost surely and in $L^2(\P)$. Since $\m\mesh(\D_n) \rightarrow 0$ almost surely, we can see that $\sigma\big(W_t : t\in[0,T]\big)\subseteq\F_\infty$ by Proposition \ref{prop:filtration}. The result now directly follows by (\ref{eq:doob_convergence}). It is worth noting that any additional information given by $\G_n$ is independent of $W$, and therefore will have no effect on the conditional expectations within our proof.
\end{proof}
\begin{remark}
If $q=0$, then $\W_n = \{W_{t_{n+1}} - W_{t_n}\}$ and each $\widetilde{W}^n$ is the piecewise linear path that agrees with the Brownian motion $W$ at the discretization points $\{t_n\}$ given by the partition $\D_n$.
More generally, $\widetilde{W}^n$ is a piecewise degree $q+1$ polynomial that matches ``moment'' information of $W$ on $\D_n$ (see \cite[Theorem 2.4]{foster2020poly} for details).
\end{remark}\smallbreak
In order to utilise results from rough path theory, we extend the above theorem to include the second iterated integrals of Brownian motion (or equivalent L\'{e}vy areas).\smallbreak

\begin{theorem}[Convergence of nested unbiased approximations for L\'{e}vy area] Let $\{\widetilde{W}^n\}$ denote the sequence of piecewise polynomial approximations for $W$ given by Theorem \ref{thm:basic_approximation}, with the same assumptions on the partitions $\{\D_n\}$ and filtration $\{\F_n\}$. We consider the following sequence of approximate L\'{e}vy areas $\{\widetilde{A}_{s,t}^n\}$ defined on $[s,t]$,\label{thm:basic_area_approximation}
\begin{align}\label{eq:approx_levy_area}
\widetilde{A}_{s,t}^n & := \bigg\{\int_s^t \big(\widetilde{W}_{s,r}^n\big)^i\, d\big(\widetilde{W}_r^n\big)^j - \int_s^t \big(\widetilde{W}_{s,r}^n\big)^j\, d\big(\widetilde{W}_r^n\big)^i\bigg\}_{1\m\leq\m i,\m j\m\leq\m d}\m,
\end{align}
where $0\leq s\leq t\leq T$ and $\widetilde{W}_{s,r}^n := \widetilde{W}_r^n - \widetilde{W}_s^n$ denotes the path increment of $\widetilde{W}^n$. Then there exists a random variable $\widetilde{A}_{s,t}$ such that $\widetilde{A}_{s,t}^n\rightarrow \widetilde{A}_{s,t}$ almost surely and in $L^2(\P)$.
Furthermore $\widetilde{A}_{s,t}$ coincides almost surely with the ```L\'{e}vy area'' of Brownian motion,
\begin{align}\label{eq:levy_area}
A_{s,t} := \bigg\{\int_s^t W_{s,r}^{\m i} \circ dW_r^{\m j} - \int_s^t W_{s,r}^{\m i}\circ dW_r^{\m j}\bigg\}_{1\m\leq\m i,\m j\m\leq\m d}\m,
\end{align}
where $W_{s,r} := W_r - W_s\m$.
\end{theorem}
\begin{proof}
When $i = j$, we see that $\big(\widetilde{A}_{s,t}^n\big)^{i,j} = A_{s,t}^{i,j} = 0$ and there is nothing to prove.
Since $\big(\widetilde{A}_{s,t}^n\big)^{i,j} = - \big(\widetilde{A}_{s,t}^n\big)^{j,i}$ and $A_{s,t}^{i,j} = - A_{s,t}^{j,i}$, it suffices to consider the case $i < j$.
Let $n\geq 0$ be fixed and suppose $t_k$ and $t_{k+1}$ are adjacent points in the partition $\D_n\m$.\medbreak

We first note the following four properties about Brownian motion:\smallbreak
\begin{enumerate}
\item Brownian motion has independent increments that are normally distributed. Thus $\{W_{t_k, t}\}_{t\in[t_k,\m t_{k+1}]}$ is independent of $\{W_{t}\}_{t\in[0,\m t_k]}$ and $\{W_{t_{k+1}, t}\}_{[t_{k+1},\m T]}\m$.\smallbreak
\item Brownian motion has independent coordinate processes. That is, when $i\neq j$, $W^i$ and $W^j$ are independent.\smallbreak

\item For $c > 0$, the process $t\mapsto \frac{1}{c} W_{c^2 t}$ has the same law as a Brownian motion.\smallbreak

\item For a fixed $q\geq 0$, the process $\big\{W_t - \E\big[W_t\m | \m\big\{\int_0^1 s^{\m p}\,dW_s\big\}_{0\m\leq\m p\m\leq\m q}\m\big]\big\}_{t\in[0, 1]}$ is a centered Gaussian process on $[0,1]$ and is independent of $\big\{\int_0^1 s^{\m p}\,dW_s\big\}_{0\m\leq\m p\m\leq\m q}$. This is a direct consequence of a polynomial expansion of Brownian motion (see, for example, \cite[Theorem 2.2]{foster2020poly} or \cite[Proposition 1.3]{habermann2021poly} for more details).\medbreak
\end{enumerate}
Since the information currently known about $W$ is given in $\F_n = \sigma\big(\mathcal{E}\big(\D_n\m, W\big)\cup \G_n\big)$ and $\D_n$ contains no further subintervals of $[t_k,\m t_{k+1}]$ when $t_k\m, t_{k+1}\in \D_n$ are adjacent, it follows from the first and fourth properties of Brownian motion that
\begin{align*}
\widetilde{W}_{t, t_k}^n = \E\bigg[W_{t_k, t} \,\Big|\, \int_{t_k}^{t_{k+1}} (t-t_k)^{\m p}\,dW_t\m,\m 0\leq p\leq q\bigg]\m,
\end{align*}
for $t\in [t_k,\m t_{k+1}]$. By applying a change of variables $r := \frac{1}{t_{k+1}-t_k}(t-t_k)$ it follows from the last three properties of Brownian motion that the $i$-th and $j$-th coordinates of
\begin{align*}
\Bigg\{\frac{W_{t_k, (t_{k+1}-t_k)r} - \E\Big[W_{t_k, (t_{k+1}-t_k)r} \,\big| \,\big\{\int_0^1 \big((t_{k+1}-t_k)r\m\big)^p\m dW_{(t_{k+1}-t_k)r}\big\}_{0\m\leq\m p\m\leq\m q}\Big]}{\sqrt{t_{k+1}-t_k}}\Bigg\}_{r\in[0, 1]}\hspace{-0.1mm},
\end{align*}
are independent centered Gaussian processes. Hence, the $i$-th and $j$-th coordinates of
\begin{align*}
\,\bigg\{\frac{1}{\sqrt{t_{k+1}-t_k}}\Big(W_{t_k, r(t_{k+1} - t_k)} - \widetilde{W}_{t_k, r(t_{k+1} - t_k)}^n\Big)\bigg\}_{r\in[0,1]}\m,
\end{align*}
are independent centered Gaussian processes. As before, we note that any additional information in $\G_n$ is independent of $W$ and does not affect conditional expectations.
Therefore, it follows that $\{W_{t_k, t}^i\}_{t\in[t_k,\m t_{k+1}]}$ and $\{W_{t_k, t}^j\}_{t\in[t_k,\m t_{k+1}]}$ are $\F_n$-conditionally independent processes and
\begin{align*}
\widetilde{W}_t^n = \E\big[W_{t_k, t} \m|\m\F_n\big].
\end{align*}
Thus for a sequence of uniform partitions $\{\triangle_m\}$ of $[s,t]$ with $\mesh(\triangle_m)\rightarrow 0$, we have
\begin{align}\label{eq:riemann_sum}
\E\bigg[\sum_{t_l\in\triangle_m}W_{s, t_l}^i W_{t_l, t_{l+1}}^j  \m\Big|\m\F_n\bigg] & = \sum_{t_l\in\triangle_m}\E\big[W_{s, t_l}^i\m|\m\F_n\big]\m\E\big[W_{t_l, t_{l+1}}^j\m|\m\F_n\big]\nonumber\\
& = \sum_{t_l\in\triangle_m}\big(\widetilde{W}_{s, t_l}^n\big)^i\big(\widetilde{W}_{t_l, t_{l+1}}^n\big)^j.
\end{align}
Since It\^{o} integrals can be defined as the limit of their Riemann sum
approximations, with convergence taking place is in the standard $L^2(\P)$ sense, it follows that
\begin{align*}
&\E\bigg[\bigg(\E\bigg[\sum_{t_l\in\triangle_m}W_{s, t_l}^i W_{t_l, t_{l+1}}^j  \m\Big|\m\F_n\bigg]-\E\bigg[\int_s^t W_{s, u}^i \, dW_u^j \m\Big|\m\F_n\bigg]\bigg)^2\m\bigg]\\
&\mm\leq \E\bigg[\bigg(\E\bigg[\sum_{t_l\in\triangle_m}W_{s, t_l}^i W_{t_l, t_{l+1}}^j\bigg]-\E\bigg[\int_s^t W_{s, u}^i \, dW_u^j\bigg]\bigg)^2\m\bigg]\rightarrow 0,
\end{align*}
by Jensen's inequality and the tower property. Since $W^i$ and $W^j$ are independent, we also observe that the It\^{o} and Stratonovich formulations of the integral coincide. That is, $\int_s^t W_{s, u}^i\m dW_u^j = \int_s^t W_{s, u}^i \circ dW_u^j\m$. Hence, by taking the limit of (\ref{eq:riemann_sum}), we have
\begin{align}\label{eq:levy_area_as_martingale}
\E\bigg[\int_s^t W_{s, u}^i \circ dW_u^j \,\Big|\,\F_n\bigg] = \int_s^t \big(\widetilde{W}_{s,u}^n\big)^i d\big(\widetilde{W}_u^n\big)^j.
\end{align}
As before, since $\{\F_n\}$ is a filtration, it follows that the sequence (\ref{eq:levy_area_as_martingale}) is a martingale. Thus, by applying Doob's martingale convergence theorem to these integrals, we have
\begin{align}\label{eq:integral_convergence}
\int_s^t \big(\widetilde{W}_{s,u}^n\big)^i d\big(\widetilde{W}_u^n\big)^j \rightarrow \E\bigg[\int_s^t W_{s, u}^i \circ dW_u^j \,\Big|\,\F_\infty\bigg] = \int_s^t W_{s, u}^i \circ dW_u^j\m,
\end{align}
as $n\rightarrow\infty$ almost surely and in $L^2(\P)$.
\end{proof}

We shall now introduce the concepts of rough paths and $\alpha\m$-H\"{o}lder regularity. These are the two key ingredients needed to establish SDE (or CDE) solutions as Lipschitz continuous functions of the Brownian motion (or driving path), see \cite{friz2024rsdes, lyons2007notes}.\smallbreak

\begin{definition}\label{def:rough_path_stuff}
A rough path is a continuous function $\boldsymbol{X} : \Delta_T \rightarrow \R\oplus \R^d\oplus \R^{d\times d}$ where $\Delta_T := \{(s,t) : 0\leq s \leq t\leq T\}$. The rough path $\boldsymbol{X} = (1, X, \mathbb{X})$ is said to be $\alpha$-H\"{o}lder continuous if $\m\|\boldsymbol{X}\|_{\alpha\text{-H\"{o}l;[0,T]}} < \infty$ where the norm $\|\cdot\|_{\alpha\text{-H\"{o}l;[s,t]}}$ is defined as
\begin{align}\label{eq:holder_norm}
\|\boldsymbol{X}\|_{\alpha\text{-H\"{o}l;[s,t]}} := \max\bigg(\sup_{s\leq u < v\leq t} \frac{\|X_{u,v}\|_2}{|v-u|^{\alpha}}\,, \sup_{s\leq u < v\leq t} \frac{\|\mathbb{X}_{u,v}\|_2}{|v-u|^{2\alpha}}\bigg),
\end{align}
is finite. Similarly, we define the $\alpha$-H\"{o}lder metric between rough paths $\boldsymbol{X}$ and $\boldsymbol{Y}$ as
\begin{align}\label{eq:holder_metric}
d_{\alpha\text{-H\"{o}l;[s,t]}}\big(\boldsymbol{X}, \boldsymbol{Y}\big) := \max\bigg(\sup_{s\leq u < v\leq t}\hspace{-0.75mm} \frac{\|X_{u,v} - Y_{u,v}\|_2}{|v-u|^{\alpha}}\,,\hspace{-0.5mm} \sup_{s\leq u < v\leq t}\hspace{-0.75mm} \frac{\|\mathbb{X}_{u,v} - \mathbb{Y}_{u,v}\|_2}{|v-u|^{2\alpha}}\bigg).
\end{align}
\end{definition}
\begin{definition} We define a ``homogeneous'' norm $\|\cdot\|$ for path increments as\label{def:homogeneous_norm}
\begin{align}\label{eq:homogeneous_norm}
\|\boldsymbol{X}\| := \max\Big(\|X\|_2, \|\mathbb{X}\|_2^\frac{1}{2}\Big),
\end{align}
where $\boldsymbol{X} = (1, X, \mathbb{X}) \in \R\oplus \R^d\oplus \R^{d\times d}$ and $\|\cdot\|_2$ denotes the standard Euclidean norm.
\end{definition}\smallbreak
Now we have introduced some key definitions from rough path theory, we can strengthen the pointwise and mean squared convergence established in Theorem \ref{thm:basic_approximation}.
To do so, we will largely follow the rough path analysis detailed in \cite[Section 13.3.2]{frizvictoir2010roughpaths}.
\smallbreak

\begin{theorem}[Rough path convergence of nested unbiased approximations]\label{thm:rough_path_convergence}
Let $\boldsymbol{\widetilde{W}}^n\hspace{-1mm} = \{\boldsymbol{\widetilde{W}}_{\hspace{-0.75mm}s,t}^n\}_{0\leq s\leq t\leq T}$ be the rough path defined by the piecewise polynomial $\widetilde{W}^n\hspace{-0.75mm}$ as
\begin{align*}
\boldsymbol{\widetilde{W}}_{\hspace{-0.75mm}s,t}^n := \big(1, \widetilde{W}_{s,t}^n\m, \mathbb{\widetilde{W}}_{s,t}^n\m\big),
\end{align*}
where $\m\widetilde{W}_{s,t}^n := \big\{(\widetilde{W}_t^n)^i - (\widetilde{W}_s^n)^i\big\}_{1\leq i\leq d}$ and $\,\mathbb{\widetilde{W}}_{s,t}^n := \big\{\int_s^t (\widetilde{W}_{s, u}^n)^i\m d(\widetilde{W}_u^n)^j\big\}_{1\leq i,j\leq d}\m$. Then
\begin{align}\label{eq:rough_path_convergence}
d_{\alpha\text{-H\"{o}l;[0,T]}}(\boldsymbol{\widetilde{W}}^n\hspace{-1mm}, \boldsymbol{W})\rightarrow 0,
\end{align}
as $n\rightarrow\infty$ almost surely, where $\alpha\in \big(\frac{1}{3}, \frac{1}{2}\big)$ and $\boldsymbol{W} = \{\boldsymbol{W}_{\hspace{-0.75mm}s,t}\}_{0\leq s\leq t\leq T}$ denotes the ``Stratonovich enhanced'' Brownian motion defined as
\begin{align*}
\boldsymbol{W}_{\hspace{-0.75mm}s,t} := \big(1, W_{s,t}\m, \mathbb{W}_{s,t}\m\big),
\end{align*}
with $\m W_{s,t} := \big\{W_t^i - W_s^i\big\}_{1\leq i\leq d}\in\R^d\m$ and $\m\mathbb{W}_{s,t} := \big\{\int_s^t W_{s, u}^i \circ dW_u^j\big\}_{1\leq i,j\leq d}\in\R^{d\times d}\m$.
\end{theorem}
\begin{proof}
Let $\beta \in (\alpha, \frac{1}{2})$ denote a fixed constant. Then by \cite[Corollary 13.14 (i)]{frizvictoir2010roughpaths}, there exists $c > 0$ (depending on $\beta$ and $T$) such that
\begin{align*}
\E\bigg[\exp\Big(c\m\|\boldsymbol{W}\|_{\beta\text{-H\"{o}l;[0,T]}}^2\Big)\bigg] <  \infty.
\end{align*}
In particular, this implies that there exists a non-negative random variable $C_1$ with $\|C_1\|_{L^2(\P)} < \infty$ so that, for almost all $\omega\in\Omega$, the increments of the path $\boldsymbol{W}(\omega)$ satisfy
\begin{align*}
\big\|\boldsymbol{W}_{\hspace{-0.75mm}s,t}\big\|^2 \leq C_1 |t-s|^{2\beta},
\end{align*}
for $0\leq s < t \leq T$ where $\|\cdot\|$ is the norm previously given by (\ref{eq:homogeneous_norm}) in Definition \ref{def:homogeneous_norm}. Hence, there exists a non-negative random variable $C_2$ with $\|C_2\|_{L^2(\P)} < \infty$ and
\begin{align}\label{eq:c2_inequality}
-C_2 |t-s|^{2\beta} \leq \int_s^t W_{s,u}^i \circ dW_{u}^j \leq C_2 |t-s|^{2\beta}.
\end{align}
Since $\E\big[C_2 | \F_n\big]$ is a martingale, we can apply Doob's maximal inequality to give
\begin{align}\label{eq:c2_bound}
\bigg\|\sup_{0\leq m\leq n}\Big(\E\big[C_2 \m|\m \F_m\big]\Big)\bigg\|_{L^2(\P)} \leq 2\big\|\E\big[C_2\m |\m \F_n\big]\big\|_{L^2(\P)} \leq 2\|C_2\|_{L^2(\P)}.
\end{align}
Taking the expectation of the inequality (\ref{eq:c2_inequality}) conditional on the $\sigma$-algebra $\F_n$ and applying equation (\ref{eq:levy_area_as_martingale}) gives
\begin{align}\label{eq:c3_bound}
-C_3 |t-s|^{2\beta} \leq \int_s^t \big(\widetilde{W}_{s,u}^n\big)^i d\big(\widetilde{W}_u^n\big)^j \leq C_3 |t-s|^{2\beta},
\end{align}
where $C_3 := \sup_{n\m\geq\m 0} \big(\E\big[C_2 \m|\m \F_n\big]\big)$ is finite almost surely due to the upper bound (\ref{eq:c2_bound}). Therefore, by setting $C_4 := d^2 C_3$, we have 
\begin{align}
\sup_{n\m\geq\m 0}\bigg\|\int_s^t \big(\widetilde{W}_{s,u}^n\big) \otimes d\big(\widetilde{W}_u^n\big)\bigg\|_2 \leq C_4 |t-s|^{2\beta}.
\end{align}
By a similar martingale argument, the exists a positive random variable $C_5$ such that
\begin{align*}
\sup_{n\m\geq\m 0}\big\|\widetilde{W}_{s,t}^n\big\|_2 \leq C_5 |t-s|^\beta,
\end{align*}
where $C_5$ is finite almost surely. Thus, letting $C_6 := \max\big(\sqrt{C_1}, \sqrt{C_4}\m,C_5\big)$, we have
\begin{align}
\big\|\boldsymbol{W}_{\hspace{-0.75mm}s,t}\big\| & \leq C_6 |t-s|^\beta,\label{eq:w_holder_bound}\\
\big\|\boldsymbol{\widetilde{W}}_{s,t}^n\big\| & = \max\bigg(\big\|\widetilde{W}_{s,t}^n\big\|_2\m, \bigg\|\int_s^t \big(\widetilde{W}_{s,u}^n\big) \otimes d\big(\widetilde{W}_u^n\big)\bigg\|_2^\frac{1}{2}\bigg) \leq C_6 |t-s|^\beta,\label{eq:poly_holder_bound}
\end{align}
for $n\geq 0$. Letting $\mathbb{W}_{s,t}$ and $\mathbb{\widetilde{W}}_{s,t}^n$ denote the second iterated integrals of $W$ and $\widetilde{W}^n$ over the interval $[s,t]$, we note that
\begin{align*}
\frac{\big\|W_{s,t} -  \widetilde{W}_{s,t}^n\big\|_2}{|t-s|^{\alpha}} & \leq \bigg(\frac{\big\|W_{s,t} -  \widetilde{W}_{s,t}^n\big\|_2}{|t-s|^{\beta}}\bigg)^\frac{\alpha}{\beta}\bigg(\sup_{0\leq s\leq t\leq T}\big\|W_{s,t} -  \widetilde{W}_{s,t}^n\big\|_2\bigg)^{1-\frac{\alpha}{\beta}},\\
\frac{\big\|\mathbb{W}_{s,t} -  \mathbb{\widetilde{W}}_{s,t}^n\big\|_2}{|t-s|^{2\alpha}} & \leq \bigg(\frac{\big\|\mathbb{W}_{s,t} -  \mathbb{\widetilde{W}}_{s,t}^n\big\|_2}{|t-s|^{2\beta}}\bigg)^\frac{\alpha}{\beta}\bigg(\sup_{0\leq s\leq t\leq T}\big\|\mathbb{W}_{s,t} -  \mathbb{\widetilde{W}}_{s,t}^n\big\|_2\bigg)^{1-\frac{\alpha}{\beta}}.
\end{align*}
By (\ref{eq:w_holder_bound}) and (\ref{eq:poly_holder_bound}), there exists an almost surely finite random variable $C_7 > 0$, which can bound the first terms on the right hand sides.
\begin{align}
\frac{\big\|W_{s,t} -  \widetilde{W}_{s,t}^n\big\|_2}{|t-s|^{\alpha}} & \leq C_7\bigg(\sup_{0\leq s\leq t\leq T}\big\|W_{s,t} -  \widetilde{W}_{s,t}^n\big\|_2\bigg)^{1-\frac{\alpha}{\beta}},\label{eq:increment_holder_bound}\\
\frac{\big\|\mathbb{W}_{s,t} -  \mathbb{\widetilde{W}}_{s,t}^n\big\|_2}{|t-s|^{2\alpha}} & \leq C_7\bigg(\sup_{0\leq s\leq t\leq T}\big\|\mathbb{W}_{s,t} -  \mathbb{\widetilde{W}}_{s,t}^n\big\|_2\bigg)^{1-\frac{\alpha}{\beta}}.\label{eq:area_holder_bound}
\end{align}
Note that it immediately follows from the uniform H\"{o}lder bounds (\ref{eq:w_holder_bound}) and (\ref{eq:poly_holder_bound}) that, almost surely, any subsequence of $\{(W - \widetilde{W}^n, \mathbb{W} - \mathbb{\widetilde{W}}^n)\}_{n\m\geq\m 0}$ is uniformly bounded and uniformly equicontinuous in $(s,t)$. So by the Arzel\`{a}-Ascoli theorem, there will exist a further subsequence converging uniformly to some random continuous function $E : \{(s,t) : 0\leq s\leq t\leq T\}\rightarrow\R^d\oplus \R^{d\times d}$. Therefore, almost surely, for a fixed $(s,t)$, $E_{s,t}$ is the limit of a subsequence of $\{(W - \widetilde{W}^n, \mathbb{W} - \mathbb{\widetilde{W}}^n)\}$. However, by Theorems \ref{thm:basic_approximation} and \ref{thm:basic_area_approximation}, the only possible limit is zero as $\boldsymbol{\widetilde{W}}{}_{\hspace{-0.75mm}s,t}^n$ converges to $\boldsymbol{W}_{\hspace{-0.75mm}s,t}$ almost surely.\vspace{0.5mm}
Since $E$ is uniformly continuous almost surely, it directly follows that it must be zero.
Hence, almost surely, $\{(W - \widetilde{W}^n, \mathbb{W} - \mathbb{\widetilde{W}}^n)\}$ will converge uniformly to zero as $n\rightarrow\infty$. \medbreak

In particular, this implies that the right hand sides of equations (\ref{eq:increment_holder_bound}) and (\ref{eq:area_holder_bound}) converge to zero almost surely -- which gives $\m d_{\alpha\text{-H\"{o}l;}[0,T]}(\boldsymbol{\widetilde{W}}{}^n, \boldsymbol{W})\rightarrow 0$ as required.
\end{proof}

\subsection{Pathwise convergence of ``polynomial-driven'' CDEs with small local errors}
Since we have shown $\big\{\boldsymbol{\widetilde{W}}^n\big\}$ converges to the Stratonovich Brownian motion $\boldsymbol{W}$ in the $\alpha\m$-H\"{o}lder metric, we can apply an extension of the \textit{Universal Limit Theorem} from rough path theory \cite[Theorem 4.9]{friz2024rsdes} to show that the sequence of CDEs (where the $n$-th CDE is driven by $\widetilde{W}^n$) converges in the $\alpha\m$-H\"{o}lder metric to the rough differential equation driven by $\boldsymbol{W}$ -- which is just the Stratonovich SDE (\ref{eq:strat_SDE}).
To this end, we use the following definition to impose smoothness on the vector fields.\smallbreak

\begin{definition}\label{def:lip_gamma} We say that a function $f:\R^e\rightarrow\R^e$ is $\Lip(\gamma)$ with $\gamma > 1$ if it is bounded with $\lfloor \gamma \rfloor$ bounded derivatives, the last being H\"{o}lder continuous with exponent $(\gamma - \lfloor \gamma \rfloor)$. We say that $f$ is $\Lip(1)$ if it is bounded and Lipschitz continuous.
\end{definition}\smallbreak

\begin{proposition}[Rough path convergence of CDEs driven by polynomials] Suppose the drift $f$ is $\Lip(1)$ and that each noise vector field $g_i$ is $\Lip(\gamma)$ for $\gamma\in(2,3)$. Let $\{\D_n\}$ be a sequence of no-skip dyadic partitions with respect to a Brownian motion $W$ (obtained using an adaptive step size control). For each (random) partition $\D_n\m$, let $\widetilde{W}^n$ denote the associated piecewise degree $q+1$ polynomial approximation that matches the increments and integrals in (\ref{eq:path_enhancement}).\label{prop:poly_CDE}
Let $y$ be the solution of the SDE (\ref{eq:strat_SDE}) and, for each $n\geq 1$, we define $\widetilde{y}^{\m n}$ as the solution of
\begin{align}
    \label{eq:poly_CDE}
    d\m\widetilde{y}_t^{\m n} & = f(\m\widetilde{y}_t^{\m n})\m dt + g(\m\widetilde{y}_t^{\m n})\m d\widetilde{W}_t^{\m n},\\[3pt]
    \widetilde{y}_0^{\m n} & = y_0\m.\nonumber 
\end{align}
Then, letting $\boldsymbol{y}$ and $\boldsymbol{\widetilde{y}}^{\m n}$ denote the unique (Stratonovich) lifts of $y$ and $\widetilde{y}^{\m n}$, we have
\begin{align}
d_{\alpha\text{-H\"{o}l;[0,T]}}(\boldsymbol{\widetilde{y}}^{\m n}\hspace{-1mm}, \boldsymbol{y})\rightarrow 0,
\end{align}
as $n\rightarrow\infty$ almost surely, where $\alpha\in\big(\gamma^{-1}, \frac{1}{2}\big)$.
\end{proposition}
\begin{proof}
The result follows from Theorem \ref{thm:rough_path_convergence} along with the Lipschitz continuity of ``Rough SDEs'' with respect to the driving rough path, see \cite[Theorem 4.9]{friz2024rsdes}.
\end{proof}\smallbreak
\begin{remark}
Technically speaking, we define the rough path $\boldsymbol{y} = (1, y, \mathbbm{y})$ as the limit of $\{\boldsymbol{\widetilde{y}}^{\m n}\}_{n\m\geq\m 0}$ and, using \cite[Theorem 4.9]{friz2024rsdes}, $\boldsymbol{y}$ can be taken to be the solution of the \textit{rough differential equation} driven by Stratonovich enhanced Brownian motion $\boldsymbol{W}$. In particular, from $\boldsymbol{y}$, we can define a process $y=\{y_t\}$ as $y_t = y_0 + y_{0,t}$ for $t\in [0,T]$, which coincides almost surely with the usual solution of the Stratonovich SDE (\ref{eq:strat_SDE}).
\end{remark}

In practice, it is unlikely that each polynomial driven CDE can be solved exactly.
However, if our numerical method approximates the CDE (\ref{eq:poly_CDE}) with $o(h_k)$ accuracy, then we would still expect convergence. For example, Heun's method (\ref{eq:heun}) is clearly obtained by discretizing the ``piecewise linear'' CDE (or Wong-Zakai approximation).\medbreak

To make this precise, we follow the error analysis presented in \cite[Section 4.2]{gaines1997variable}.
However, as rough path theory was formulated after \cite{gaines1997variable}, we can use an extension of \textit{Davie's estimate} \cite[Theorem 10.29]{frizvictoir2010roughpaths} to establish a Lipschitz property of CDE flows.

\begin{proposition}[CDEs driven by Brownian polynomials have Lipschitz flows] Let $\{\widetilde{W}^{\m n}\}$ denote the piecewise polynomial approximations of $W$ corresponding to the sequence of no-skip dyadic partitions $\{\D_n\}$. For each $n\geq 1$, we consider solutions $\widetilde{x}^{\m n}$ and $\widetilde{y}^{\m n}$ of the CDE (\ref{eq:poly_CDE}), driven by the same $\widetilde{W}^{\m n}\hspace{-0.25mm}$, but with different initial values.
Then there exists a positive random variable $L$ with $L < \infty$ almost surely, such that\label{prop:davie_estimate}
\begin{align}\label{eq:flow_estimate}
\|\m\widetilde{x}_t^{\m n} - \widetilde{y}_t^{\m n}\|_2\leq L\m \|\m\widetilde{x}_s^{\m n} - \widetilde{y}_s^{\m n}\|_2\m,
\end{align}
for all $\m 0\leq s\leq t\leq T\m$ and $\m n\geq 0$.
\end{proposition}
\begin{proof}
We first note that $\{\|\widetilde{W}^{\m n}\|_{\alpha\text{-H\"{o}l;}[0,T]}\}_{n\geq 0}$ can be uniformly bounded by an almost surely finite positive random variable (as shown in the proof of Theorem \ref{thm:rough_path_convergence}).
Therefore, as $f$ and $g$ are both Lipschitz, we can apply an extension of \textit{Davie's estimate} (given by \cite[Theorem 4.9]{friz2024rsdes}) to the two CDE solutions, $\widetilde{x}^{\m n}$ and $\widetilde{y}^{\m n}$, since the condition
\begin{align*}
\|\widetilde{W}^{\m n}\|_{\alpha\text{-H\"{o}l;}[0,T]} + \|f\|_{\Lip(1)} + \|g\|_{\Lip(1)} \leq M,
\end{align*}
holds for some almost surely finite random variable $M$. The result directly follows by applying \cite[Theorem 4.9]{friz2024rsdes} to the CDE solutions $\widetilde{x}^{\m n}$ and $\widetilde{y}^{\m n}$, with $\|\cdot\|_m = \|\cdot\|_2$.
\end{proof}

Using the Lipschitz estimate (\ref{eq:flow_estimate}), we now adapt the proof of \cite[Theorem 4.3]{gaines1997variable} to give our main result on pathwise convergence for adaptive approximations close to polynomial driven CDEs -- such as Heun's method (\ref{eq:heun}) and the SRK method (\ref{eq:srk}).\smallbreak

\begin{theorem}[Pathwise convergence of SDE solvers with adaptive step sizes] For each partition $\D_n\m$, appearing in Propositions \ref{prop:poly_CDE} and \ref{prop:davie_estimate}, let $Y^n = \{Y_k^n\}$ be a numerical solution of (\ref{eq:strat_SDE}) computed at times $\{0 = t_0^n < t_1^n < \cdots < t_{K_n}^n = T\} = \D_n\m$, with the initial condition $Y_0^n := y_0$ such that, for all $k\in\{0, 1,\cdots, K_n - 1\}$, we have\label{thm:pathwise_conv_thm}
\begin{align}\label{eq:near_cde}
\big\|\m Y_{k+1}^n  - \Phi_{t_k^n\m,t_{k+1}^n}^n(Y_k^n)\big\|_2 \leq w(t_k^n, t_{k+1}^n),
\end{align}
where $\Phi_{a\m,b}^n\hspace{-0.5mm} : \R^e\rightarrow\R^e$ denotes the solution $\widetilde{x}_b^{\m n}$ at time $b$ of the following CDE on $[a,b]$,
\begin{align}
    \label{eq:poly_CDE_map}
    d\widetilde{x}_t^{\m n} & = f(\widetilde{x}_t^{\m n})\m dt + g(\widetilde{x}_t^{\m n})\m d\widetilde{W}_t^{\m n},\\[3pt]
    \widetilde{x}_a^{\m n} & := y,\nonumber 
\end{align}
driven by the same piecewise polynomial $\{\widetilde{W}_t^{\m n}\}_{t\in[0,T]}$ considered in Proposition \ref{prop:poly_CDE} and
\begin{align}\label{eq:small_errors}
\sum_{k = 0}^{K_n - 1} w(t_k^n, t_{k+1}^n) \rightarrow 0,
\end{align}
as $n\rightarrow\infty$ almost surely. Then the numerical solutions $\{Y^n\}$ converges pathwise, i.e.
\begin{align}\label{eq:pathwise_conv}
\sup_{0\m\leq\m k\m\leq\m K_n}\|Y_k^n - y_{t_k^n}\|_2\rightarrow 0,
\end{align}
as $n\rightarrow\infty$ almost surely.
\end{theorem}
\begin{proof}
Since $\Phi_{t_k^n, t_k^n}^n (Y_k^n) = Y_k^n$ and $(t_0^n, Y_0^n) = (0, y_0)$ for all $n\geq 0$, by expressing $Y_k^n - \Phi_{0, t_k^n}^n (y_0)\m$ as a telescoping sum and applying the triangle inequality, we obtain
\begin{align*}
\big\|\m Y_k^n - \Phi_{0, t_k^n}^n (y_0)\big\|_2 & \leq \sum_{i=0}^{k-1} \big\|\Phi_{t_{i+1}^n, t_k^n}^n (Y_{i+1}^n) - \Phi_{t_i^n, t_k^n}^n (Y_i^n)\big\|_2\m.
\end{align*}
As $\Phi^n$ denotes the solution of a differential equation, it has a natural flow property, $\Phi_{s,t}^n = \Phi_{u,t}^n\circ \Phi_{s,u}^n$ for $u\in[s,t]$. In particular, by applying this to $\Phi_{t_i^n, t_k^n}^n (Y_i^n)$, we have
\begin{align*}
\big\|\m Y_k^n - \Phi_{0, t_k^n}^n (y_0)\big\|_2 \leq \sum_{i=0}^{k-1} \big\|\Phi_{t_{i+1}^n, t_k^n}^n (Y_{i+1}^n) - \Phi_{t_{i+1}^n, t_k^n}^n\big(\Phi_{t_i^n, t_{i+1}^n}^n(Y_i^n)\big)\big\|_2\m.
\end{align*}
Thus, applying the Lipschitz estimate (\ref{eq:flow_estimate}) from Proposition \ref{prop:davie_estimate} to each $\Phi_{t_{i+1}^n, t_k^n}^n\m$,
\begin{align*}
\hspace{20mm}\big\|\m Y_k^n - \Phi_{0, t_k^n}^n (y_0)\big\|_2 & \leq \sum_{i=0}^{k-1} L\m\big\|Y_{i+1}^n - \Phi_{t_i^n, t_{i+1}^n}^n(Y_i^n)\big\|_2 \\
&\leq L\sum_{i=0}^{k-1}w(t_i^n, t_{i+1}^n),\hspace{20mm}[\m\text{by }(\ref{eq:near_cde})].
\end{align*}
By the assumption (\ref{eq:small_errors}) that the ``error bounds'' $\big\{w(t_i^n, t_{i+1}^n)\big\}$ are small, we see that
\begin{align}
\sup_{0\leq k\leq K_n}\big\|\m Y_k^n - \Phi_{0, t_k^n}^n (y_0)\big\|_2 \leq L\sum_{i=0}^{K_n-1}w(t_i^n, t_{i+1}^n)\rightarrow 0,
\end{align}
as $n\rightarrow\infty$ almost surely. The result now immediately follows by Theorem \ref{thm:rough_path_convergence}, which established the rough path convergence of $\{\Phi_{0, t}^n (y_0)\}$ to the Stratonovich solution.
\end{proof}\smallbreak
\begin{remark}
If the local errors can be bounded by $w(t_k^n, t_{k+1}^n) = \widetilde{w}(t_{k+1}^n - t_k^n)$ for some function $\widetilde{w}$ with $\widetilde{w}(h)\sim o(h)$, then $\sum_{k = 0}^{K_n - 1} w(t_k^n, t_{k+1}^n) \rightarrow 0$ is trivially satisfied.
This will be the case for standard numerical methods, such as Heun's method (\ref{eq:heun}), where under suitable regularity conditions, we have $\widetilde{w}(h)\sim O(h^{\frac{3}{2}})$ in an $L^2(\P)$ sense.
\end{remark}\smallbreak

\subsection{Numerical methods for SDEs viewed as approximations of CDEs}\label{sect:methods_as_cdes}

To conclude the section, we will argue why the convergence criterion (\ref{eq:near_cde}) is satisfied the no$\m$-area Milstein, Heun and stochastic Runge$\m$-Kutta methods given in Section \ref{sect:examples_of_methods}.
By considering Taylor expansions, we can show these numerical methods accurately approximate the CDE (\ref{eq:poly_CDE_map}) so that our main result (Theorem \ref{thm:pathwise_conv_thm}) can be applied.

\begin{theorem}[Taylor expansion of CDEs driven by ``Brownian'' polynomials] Consider the controlled differential equation considered in Theorem \ref{thm:pathwise_conv_thm} and given by\label{thm:cde_expansion}
\begin{align}\label{eq:approx_cde}
d\widetilde{x}_t^{\m n} = f(\widetilde{x}_t^{\m n})\m dt + g(\widetilde{x}_t^{\m n})\m d\widetilde{W}_t^{\m n},
\end{align}
where $f:\R^e\rightarrow\R^e$ and $g:\R^e\rightarrow\R^{e\times d}$ are once and twice continuously differentiable respectively. Suppose further that $f$ and $g$ are bounded and with bounded derivatives.
Then the solution $\widetilde{x}^{\m n}$ of the controlled differential equation (\ref{eq:approx_cde}) can be expanded as
\begin{align}\label{eq:cde_expansion}
\mm\widetilde{x}_{t_{k+1}}^{\m n} = \widetilde{x}_{t_k}^{\m n} + f(\widetilde{x}_{t_k}^{\m n})h_k + g(\widetilde{x}_{t_k}^{\m n})W_k + g^{\m\prime}(\widetilde{x}_{t_k}^{\m n})\m g(\widetilde{x}_{t_k}^{\m n})\int_{t_k}^{t_{k+1}}\hspace*{-1mm}\int_{t_k}^t d\widetilde{W}_s^{\m n}\otimes d\widetilde{W}_t^{\m n} + R,
\end{align}
where, for any fixed $\alpha\in\big(0, \frac{1}{2}\big)$, there is an almost surely finite random variable $C > 0$, not depending on $h_k$ or $\widetilde{W}^n$, such that the remainder term $R$ satisfies $\|R\|_2 \leq C h_k^{3\alpha}$.
\end{theorem}
\begin{proof}
The solution of the CDE (\ref{eq:approx_cde}) can be expressed in integral form as
\begin{align*}
\widetilde{x}_{t_{k+1}}^{\m n} = \widetilde{x}_{t_k}^{\m n} + \int_{t_k}^{t_{k+1}} \hspace{-1mm} f(\widetilde{x}_t^{\m n})\m dt + \int_{t_k}^{t_{k+1}} \hspace{-1mm} g(\widetilde{x}_t^{\m n})\m d\widetilde{W}_t^{\m n},
\end{align*}
and, given a differentiable function $F\in\mathcal{C}^1(\R^e)$, we also have the following chain rule:
\begin{align*}
    F(\widetilde{x}_t^{\m n}) = F(\widetilde{x}_{t_k}^{\m n}) + \int_{t_k}^t F^{\m\prime}(\widetilde{x}_s^{\m n})\,  d\widetilde{x}_s^{\m n}.
\end{align*}
Therefore, using the above equations, we can derive a Taylor expansion for the CDE.
\begin{align*}
\widetilde{x}_{t_{k+1}}^{\m n} & = \widetilde{x}_{t_k}^{\m n} +  \int_{t_k}^{t_{k+1}} \hspace{-1mm}  f(\widetilde{x}_{t_k}^{\m n}) + \int_{t_k}^t f^{\m\prime}(\widetilde{x}_s^{\m n})\m d\widetilde{x}_s^{\m n}\m dt + \int_{t_k}^{t_{k+1}} \hspace{-1mm} g(\widetilde{x}_{t_k}^{\m n}) + \int_{t_k}^t g^{\m\prime}(\widetilde{x}_s^{\m n})\m d\widetilde{x}_s^{\m n}\m d\widetilde{W}_t^{\m n}\\[3pt]
& = \widetilde{x}_{t_k}^{\m n} + f(\widetilde{x}_{t_k}^{\m n})h_k + g(\widetilde{x}_{t_k}^{\m n})\big(\widetilde{W}_{t_{k+1}}^{\m n} - \widetilde{W}_{t_k}^{\m n}\big)\\[2pt]
&\mm + \int_{t_k}^{t_{k+1}}\hspace*{-1mm}\int_{t_k}^t f^{\m\prime}(\widetilde{x}_s^{\m n})f(\widetilde{x}_s^{\m n})\m\m ds\m\m dt + \int_{t_k}^{t_{k+1}}\hspace*{-1mm}\int_{t_k}^t f^{\m\prime}(\widetilde{x}_s^{\m n})g(\widetilde{x}_s^{\m n})\m\m d\widetilde{W}_s^{\m n}\m dt\\
&\mm + \int_{t_k}^{t_{k+1}}\hspace*{-1mm}\int_{t_k}^t f^{\m\prime}(\widetilde{x}_s^{\m n})g(\widetilde{x}_s^{\m n})\m\m ds\m\m d\widetilde{W}_t^{\m n} + \int_{t_k}^{t_{k+1}}\hspace*{-1mm}\int_{t_k}^t g^{\m\prime}(\widetilde{x}_s^{\m n})g(\widetilde{x}_s^{\m n})\m\m d\widetilde{W}_s^{\m n}\m d\widetilde{W}_t^{\m n}\\
& = \widetilde{x}_{t_k}^{\m n} + f(\widetilde{x}_{t_k}^{\m n})h_k + g(\widetilde{x}_{t_k}^{\m n})W_k + g^{\m\prime}(\widetilde{x}_{t_k}^{\m n})\m g(\widetilde{x}_{t_k}^{\m n})\int_{t_k}^{t_{k+1}}\hspace*{-1mm}\int_{t_k}^t d\widetilde{W}_s^{\m n}\otimes d\widetilde{W}_t^{\m n} + R,
\end{align*}
where the remainder term is $R := \int_{t_k}^{t_{k+1}}\hspace*{-1mm}\int_{t_k}^t \big(g^{\m\prime}(\widetilde{x}_s^{\m n})g(\widetilde{x}_s^{\m n}) - g^{\m\prime}(\widetilde{x}_{t_k})g(\widetilde{x}_{t_k}^{\m n})\big)\m d\widetilde{W}_s^{\m n}\m d\widetilde{W}_t^{\m n}+$ $\int_{t_k}^{t_{k+1}}\hspace*{-1mm}\int_{t_k}^t \hspace{-0.25mm}f^{\m\prime}(\widetilde{x}_s^{\m n})g(\widetilde{x}_s^{\m n})\m\m d\widetilde{W}_s^{\m n}\m dt\m +\m \int_{t_k}^{t_{k+1}}\hspace*{-1mm}\int_{t_k}^t\hspace{-0.25mm} f^{\m\prime}(\widetilde{x}_s^{\m n})g(\widetilde{x}_s^{\m n})\m\m ds\m\m d\widetilde{W}_t^{\m n}\m +\m\int_{t_k}^{t_{k+1}}\hspace*{-1mm}\int_{t_k}^t\hspace{-0.25mm} f^{\m\prime}(\widetilde{x}_s^{\m n})f(\widetilde{x}_s^{\m n})\m\m ds\m\m dt$.\medbreak

As $\widetilde{W}^{\m n}$ is a polynomial on $[t_k\m,t_{k+1}]$, we can apply the chain rule to rewrite any ``$d\widetilde{W}_t^{\m n}\m$'' integral as an ``$dt\m$'' integral. This allows the first integral in $R$ to be estimated,
\begin{align*}
&\bigg\|\int_{t_k}^{t_{k+1}}\hspace*{-1mm}\int_{t_k}^t \big(g^{\m\prime}(\widetilde{x}_s^{\m n})g(\widetilde{x}_s^{\m n}) - g^{\m\prime}(\widetilde{x}_{t_k})g(\widetilde{x}_{t_k}^{\m n})\big)\m d\widetilde{W}_s^{\m n}\m d\widetilde{W}_t^{\m n}\bigg\|_2\\
&\mm\leq\int_{t_k}^{t_{k+1}}\hspace*{-1mm}\int_{t_k}^t \big\|\big(g^{\m\prime}(\widetilde{x}_s^{\m n})g(\widetilde{x}_s^{\m n}) - g^{\m\prime}(\widetilde{x}_{t_k})g(\widetilde{x}_{t_k}^{\m n})\big)\m \big(\widetilde{W}_s^{\m n}\big)^\prime \big(\widetilde{W}_t^{\m n}\big)^\prime \big\|_2\m ds\m dt\\
&\mm\leq\int_{t_k}^{t_{k+1}}\hspace*{-1mm}\int_{t_k}^t \big\|g^{\m\prime}(\widetilde{x}_s^{\m n})g(\widetilde{x}_s^{\m n}) - g^{\m\prime}(\widetilde{x}_{t_k})g(\widetilde{x}_{t_k}^{\m n})\big\|\big\|\big(\widetilde{W}_s^{\m n}\big)^\prime\m\big\|_2\big\| \big(\widetilde{W}_t^{\m n}\big)^\prime\m\big\|_2\m ds\m dt\\
&\mm\leq \frac{1}{2}h_k^2  \big\|\m g^{\m \prime}(\cdot)g(\cdot)\big\|_{\Lip(1)}\sup_{t\in[t_k\m, t_{k+1}]}\big\|\m\widetilde{x}_t^{\m n} - \widetilde{x}_{t_k}^{\m n}\big\|_2 \sup_{t\in[t_k\m, t_{k+1}]}\big\|\big(\widetilde{W}_t^{\m n}\big)^\prime\m\big\|_2^2\m.
\end{align*}
We estimate the term involving $\m\widetilde{x}^{\m n}$ using the integral form of the CDE solution as
\begin{align*}
\big\|\m\widetilde{x}_t^{\m n} - \widetilde{x}_{t_k}^{\m n}\big\|_2 & \leq \int_{t_k}^t \hspace{-1mm}\big\| f(\widetilde{x}_s^{\m n})\big\|_2\m ds + \int_{t_k}^t \hspace{-1mm} \big\|g(\widetilde{x}_s^{\m n})\big\|_2 \m \big\|\big(\widetilde{W}_s^{\m n}\big)^\prime\m\big\|_2\m ds\\
& \leq \Big(\|f\|_\infty  + \|g\|_\infty \sup_{t\in[t_k\m, t_{k+1}]}\big\|\big(\widetilde{W}_t^{\m n}\big)^\prime\m\big\|_2\Big)h_k\m.
\end{align*}
As shown in the proof of Theorem \ref{thm:rough_path_convergence}, the piecewise polynomials $\{\widetilde{W}^{\m n}\}$ are uniformly $\alpha$-H\"{o}lder continuous. That is, there exists an almost surely finite random variable $c > 0$ (independent of $n$) such that $\|\widetilde{W}_t^{\m n} - \widetilde{W}_s^{\m n}\|_2 \leq c (t-s)^\alpha$ for all $s,t\in[0,T]$.

Hence, for $m\geq 1$, we can simply bound the following iterated integrals of $\widetilde{W}^n$ as
\begin{align}\label{eq:interated_w_integral}
&\bigg\|\frac{1}{h_k^m}\int_{t_k}^{t_{k+1}} \int_{t_k}^{r_1}  \cdots \int_{t_k}^{r_{m-1}}\big(\m\widetilde{W}_{r_m}^{\m n} - \widetilde{W}_{t_k}^{\m n}\m\big)\, dr_m\cdots dr_1\bigg\|_2\\
&\mm\leq \frac{1}{h_k^m}\int_{t_k}^{t_{k+1}} \int_{t_k}^{r_1}  \cdots \int_{t_k}^{r_{m-1}}  \big\|\widetilde{W}_{r_m}^{\m n} - \widetilde{W}_{t_k}^{\m n}\big\|_2\, dr_m\cdots dr_1\nonumber\\
&\mm \leq  \frac{1}{h_k^m}\int_{t_k}^{t_{k+1}} \int_{t_k}^{r_1}  \cdots \int_{t_k}^{r_{m-1}}  c (r_m-t_k)^\alpha\, dr_m\cdots dr_1\nonumber\\
&\mm = \frac{\Gamma(\alpha+m+1)}{\Gamma(\alpha+1)}c h_k^{\alpha}.\nonumber
\end{align}
In \cite{foster2020poly}, it was shown that the coefficients in the polynomial approximation $\widetilde{W}^{\m n}$ can be expressed as linear combinations of rescaled iterated integrals with the form (\ref{eq:interated_w_integral}). In particular, this means that they can be uniformly bounded as $O(h_k^\alpha)$ almost surely. For example, in the simplest cases where the polynomial's degree is either $1$ or $2$, we have $\widetilde{W}_{t_k\m, t}^{\m n} = \frac{t-t_k}{h_k}\m W_k$ or $\widetilde{W}_{t_k\m, t}^{\m n} = \frac{t-t_k}{h_k}\m W_k + \frac{6(t_{k+1}-t)(t-t_k)}{h_k^2}\m H_k$ with $W_k\m, H_k\sim O(h_k^\alpha)$.\smallbreak

Therefore, the derivative of the polynomial satisfies $\big(\widetilde{W}^{\m n}\big)^\prime \sim O(h_k^{\alpha - 1})$ and thus
\begin{align*}
&\bigg\|\int_{t_k}^{t_{k+1}}\hspace*{-1mm}\int_{t_k}^t \big(g^{\m\prime}(\widetilde{x}_s^{\m n})g(\widetilde{x}_s^{\m n}) - g^{\m\prime}(\widetilde{x}_{t_k})g(\widetilde{x}_{t_k}^{\m n})\big)\m d\widetilde{W}_s^{\m n}\m d\widetilde{W}_t^{\m n}\bigg\|_2\\
&\leq \frac{1}{2}h_k^3  \big\|\m g^{\m \prime}g\m\big\|_{\Lip(1)} \Big(\|f\|_\infty  + \|g\|_\infty \sup_{t\in[t_k\m, t_{k+1}]}\big\|\big(\widetilde{W}_t^{\m n}\big)^\prime\m\big\|_2\Big) \sup_{t\in[t_k\m, t_{k+1}]}\big\|\big(\widetilde{W}_t^{\m n}\big)^\prime\m\big\|_2^2\sim O(h^{3\alpha}).
\end{align*}
As $\big(\widetilde{W}^{\m n}\big)^\prime\hspace{-1mm}\sim O(h_k^{\alpha - 1})$, it is straightforward to show the middle terms in $R$ are $O(h_k^{\alpha + 1})$. Finally, the fourth term in $R$ is clearly $O(h_k^2)$ due to the boundedness of $f$ and $f^\prime$.
\end{proof}\smallbreak
\begin{remark}
If we set $\alpha > \frac{1}{3}\m$, then $R$ is $o(h_k)$ as required in the condition (\ref{eq:intro_ode_condition}).
\end{remark}\smallbreak
\begin{remark}
By equation (\ref{eq:levy_area_as_martingale}), which was shown in the proof of Theorem \ref{thm:basic_area_approximation}, we see that the polynomial iterated integrals are optimal unbiased approximations as
\begin{align}\label{eq:poly_integrals_as_exps}
\int_{t_k}^{t_{k+1}}\hspace*{-1mm}\int_{t_k}^t d\widetilde{W}_s^{\m n}\otimes d\widetilde{W}_t^{\m n} = \E\bigg[\int_{t_k}^{t_{k+1}} W_{t_k\m, t} \otimes \circ\, dW_t \,\Big|\,\W_k\bigg].
\end{align}
Therefore, it is clear that the convergence condition (\ref{eq:intro_ode_condition}) originally presented in the introduction simply follows from Theorems \ref{thm:pathwise_conv_thm} and \ref{thm:cde_expansion} along with equation (\ref{eq:poly_integrals_as_exps}).
\end{remark}
\begin{remark}
Whilst a general formula for polynomial approximations of second iterated integrals (or equivalent L\'{e}vy areas) is given in \cite[Theorem 5.4]{foster2023levyarea}, we only use
\begin{align}
\E\bigg[\int_{t_k}^{t_{k+1}} W_{t_k\m, t} \otimes \circ\, dW_t \,\Big|\,W_k\bigg]& = \frac{1}{2}W_k^{\otimes 2},\label{eq:expected_integral_3}\\
\E\bigg[\int_{t_k}^{t_{k+1}} W_{t_k\m, t} \otimes \circ\, dW_t \,\Big|\,W_k\m, H_k\bigg] & = \frac{1}{2}W_k^{\otimes 2} + H_k\otimes W_k - W_k\otimes H_k\m.\label{eq:expected_integral_4}
\end{align}
\end{remark}

With the Taylor expansion given by Theorem \ref{thm:cde_expansion}, it is now straightforward to see why our analysis is applicable to the numerical methods discussed in Section \ref{sect:examples_of_methods}.
For each of these numerical methods, we shall compute their Taylor expansions and show that the resulting remainder term is sufficiently small (i.e.~$o(h_k)$ almost surely).\medbreak

\begin{itemize}[leftmargin=1.5em]
\item \underline{The ``no$\m$-area'' Milstein method (\ref{eq:milstein})}\smallbreak

As discussed in Remark \ref{rmk:milstein}, by using the It\^{o}$\m$-Stratonovich correction, the no$\m$-area Milstein method can be rewritten as the following method for Stratonovich SDEs:
\begin{align*}
Y_{k+1} := Y_k + \widetilde{f}\big(Y_k\big) h_k + g\big(Y_k\big) W_k + \frac{1}{2}g^{\m\prime}\big(Y_k\big)g\big(Y_k\big)W_k^{\otimes 2},
\end{align*}
where $\widetilde{f}(\cdot) = f(\cdot) - \frac{1}{2}\sum_{i=1}^d g_i^\prime(\cdot)g_i(\cdot)$ is the drift for the SDE in Stratonovich form.
By (\ref{eq:expected_integral_3}), it is clear that the no$\m$-area Milstein method falls under our framework.\medbreak

\item \underline{Heun's method (\ref{eq:heun})}\smallbreak

Our strategy is simply to show that Heun's method is close to no$\m$-area Milstein. 
Using the Taylor expansion $\m F(b) = F(a) + F^{\m\prime}(a)c + \int_0^1 (1-r) F^{\m\prime\prime}(a + rc)(c,c)\m dr\m$ for $a,b,c\in \R^e$ where $c = b-a$, and the notation $F^{\m\prime\prime}(\cdot)(u\otimes v) \equiv F^{\m\prime\prime}(\cdot)(u,v)$, we have
\begin{align*}
Y_{k+1} = Y_k & + \frac{1}{2}\big(F\big(Y_k\big) + F\big(\widetilde{Y}_{k+1}\big)\big)\overline{W}_{\n k}\\
= Y_k & + \frac{1}{2}F\big(Y_k\big)\overline{W}_{\n k} + \frac{1}{2}\big(F\big(Y_k\big) + F^{\m\prime}\big(Y_k\big)\big(\widetilde{Y}_{k+1} - Y_k\big) + R\big)\m\overline{W}_{\n k}\\
= Y_k & + F\big(Y_k\big)\overline{W}_{\n k} + \frac{1}{2}\big(F^{\m\prime}\big(Y_k\big)F\big(Y_k\big)\overline{W}_{\n k} + R\big)\m\overline{W}_{\n k}\\
= Y_k & + F\big(Y_k\big)\overline{W}_{\n k} +  F^{\m\prime}\big(Y_k\big)F\big(Y_k\big)\bigg(\hspace*{-17mm}\underbrace{\frac{1}{2}\overline{W}_{\n k}^{\m\otimes 2}}_{=\,\E\big[\int_{t_k}^{t_{k+1}}(\m\overline{W}_{\n t} - \overline{W}_{\n t_k})\,\otimes \,\circ\, d\overline{W}_{\n t}\,|\, \overline{W}_{\n k}\big]}\hspace*{-17mm}\bigg) + \frac{1}{2}\m R\m\overline{W}_{\n k}\m,
\end{align*}
where the remainder is $R = \int_0^1 (1-r)F^{\m\prime\prime}(Y_k + (1-r)F(Y_k)\overline{W}_{\n k})\m dr\m(F(Y_k)\overline{W}_{\n k})^{\otimes 2}$.\medbreak

Since $W$ is $\alpha$-H\"{o}lder continuous for $\alpha\in(0,\frac{1}{2})$, we can employ a similar argument as the proof of Theorem \ref{thm:cde_expansion} to show that $R\sim O(h_k^{2\alpha})$, and thus $R\m\overline{W}_{\n k}\sim O(h_k^{3\alpha})$.
Therefore, as the vector fields and their derivatives are bounded, by taking $\alpha > \frac{1}{3}\m$, we see that Heun's method has the desired expansion (up to a $o(h_k)$ remainder).\medbreak

\item \underline{SPaRK method (\ref{eq:srk})}\footnote{\textbf{S}plitting \textbf{Pa}th \textbf{R}unge$\m$-\textbf{K}utta method}\smallbreak
Taylor expanding this numerical method is slightly more involved, but results in
\begin{align*}
Y_{k+1} = Y_k & + F\big(Y_k\big)\big(a\m\overline{W}_{\n k} + \overline{H}_k\big) + b\m F\big(Y_{k+\frac{1}{2}}\big)\overline{W}_{\n k} + F\big(Z_{k+1}\big)\big(a\m\overline{W}_{\n k} - \overline{H}_k\big)\\[2pt]
= Y_k & + F\big(Y_k\big)\big(a\m\overline{W}_{\n k} + \overline{H}_k\big) + b\big(F\big(Y_k\big) + F^{\m\prime}\big(Y_k\big)\big(Y_{k+\frac{1}{2}} - Y_k\big) + R_1\big)\big)\overline{W}_{\n k}\\
& + \big(F\big(Y_k\big) + F^{\m\prime}\big(Y_k\big)\big(Z_{k+1} - Y_k\big) + R_2\big)\big(a\m\overline{W}_{\n k} - \overline{H}_k\big)\\
= Y_k & + F\big(Y_k\big)\overline{W}_{\n k} + b\m F^{\m\prime}\big(Y_k\big)F\big(Y_k\big)\bigg(\frac{1}{2}\overline{W}_{\n k} + \sqrt{3}\,\overline{H}_k\bigg)\otimes \overline{W}_{\n k} + b\m R_1 \overline{W}_{\n k}\\
& + F^{\m\prime}\big(Y_k\big)F\big(Y_k\big)\big(\overline{W}_{\n k}\otimes\big(a\m\overline{W}_{\n k} - \overline{H}_k\big)\big)  + (R_2 + R_3)\big(a\m\overline{W}_{\n k} - \overline{H}_k\big)\\
= Y_k & + F\big(Y_k\big)\overline{W}_{\n k} + F^{\m\prime}\big(Y_k\big)F\big(Y_k\big)\bigg(\hspace*{-1mm}\underbrace{\frac{1}{2}\overline{W}_{\n k}^{\m\otimes 2} + \overline{H}_{\n k}\otimes \overline{W}_{\n k} - \overline{W}_{\n k} \otimes \overline{H}_{\n k}}_{=\,\E\big[\int_{t_k}^{t_{k+1}}(\m\overline{W}_{\n t} - \overline{W}_{\n t_k})\,\otimes \,\circ\, d\overline{W}_{\n t}\,|\, \overline{W}_{\n k}\m,\m \overline{H}_{\n k}\big]}\hspace{-1mm}\bigg) + R,
\end{align*}
where $R$ is given by $R := b\m R_1 \overline{W}_{\n k} + (R_2 + R_3)\big(a\m\overline{W}_{\n k} - \overline{H}_k\big)$ with $a := \frac{3-\sqrt{3}}{6}$, $b := \frac{\sqrt{3}}{3}$, 
\begin{align*}
R_1 & = \int_0^1 (1-r)F^{\m\prime\prime}\big(Y_k + (1-r)(Y_{k+\frac{1}{2}} - Y_k)\big) dr\m\big(Y_{k+\frac{1}{2}} - Y_k\big)^{\otimes 2},\\
R_2 & = \int_0^1 (1-r)F^{\m\prime\prime}\big(Y_k + (1-r)(Z_{k+1} - Y_k)\big) dr\m\big(Z_{k+1} - Y_k\big)^{\otimes 2},\\[3pt]
R_3 & = \big(F\big(Y_{k+\frac{1}{2}}\big) - F\big(Y_k\big)\big)\overline{W}_{\n k}\m.
\end{align*}
Recall that for $\alpha\in(0,\frac{1}{2})$, Brownian motion is $\alpha$-H\"{o}lder continuous almost surely. That is, there exists an almost surely finite random variable $c > 0$ such that $\|W_{s,t}\|_2 \leq c|t-s|^{\alpha}$. Hence, space-time L\'{e}vy area is also $\alpha$-H\"{o}lder continuous as
\begin{align*}
\|H_{s,t}\|_2 = \bigg\|\m\frac{1}{t-s} \int_s^t W_{s,u}\m du - \frac{1}{2}W_{s,t}\bigg\|_2 \leq c\bigg(\frac{1}{1+\alpha} + \frac{1}{2}\bigg) |t-s|^{\alpha}.
\end{align*}
Using the $\alpha$-H\"{o}lder estimates for $W_k$ and $H_k$ along with the boundedness of $F, F^{\m\prime}$ and $F^{\m\prime\prime}$, it is straightforward to estimate each remainder term as $\|R_i\|_2 \leq c_i h_k^{2\alpha}$.
We can produce estimates for $R_1\m, R_2\sim O(h_k^{2\alpha})$ using similar arguments as before.
The third term $R_3$ can simply be estimated using the Lipschitz continuity of $F$ as
\begin{align*}
\|R_3\|_2 & \leq \big\|F\big(Y_{k+\frac{1}{2}}\big) - F\big(Y_k\big)\big\|_2\big\|\m\overline{W}_{\n k}\big\|_2\\ 
& \leq \|F\|_{\Lip(1)}\|Y_{k+\frac{1}{2}} - Y_{k}\|_2\big\|\m\overline{W}_{\n k}\big\|_2\\
& \leq \|F\|_{\Lip(1)}\|F\|_\infty \Big(\frac{1}{2}\big\|\m\overline{W}_{\n k}\big\|_2 + \sqrt{3}\m\big\|\m\overline{H}_{\n k}\big\|_2\Big)\big\|\m\overline{W}_{\n k}\big\|_2\m.
\end{align*}
Thus, it follows that the remainder $R$ is $O(h_k^{3\alpha}$), which is $o(h_k)$ if we take $\alpha > \frac{1}{3}$.
\end{itemize}

\section{Counterexample involving non-nested step sizes}\label{sect:counterexample} In this brief section, we will present a simple counterexample showing that adaptive step sizes without the ``no skip'' property can prevent the numerical approximation from converging to the desired SDE solution.
The idea is simply that by selectively skipping over values of the Brownian path, we can induce a local $O(h)$ bias with a positive expectation. These biases can then propagate linearly to produce a global $O(1)$ bias, also with a positive expectation.
Our counterexample to illustrate this is given by Theorem \ref{thm:counterexample}.
\begin{theorem}\label{thm:counterexample}
For $N\geq 1$, we will define an approximation $\{(X_k\m, Y_k)\}_{0\leq k\leq N}$ of
\begin{align}\label{eq:counterexample_sde}
dx_t & = dW_t^1\m,\nonumber\\
dy_t & = x_t \circ dW_t^2\m,\\
(x_0\m, y_0) & = 0,\nonumber
\end{align}
over $[0,T]$ by $(X_0\m, Y_0) := 0$ and for each $k\geq 0$, the numerical solution is propagated using either one step or two steps of Heun's method (depending on which is largest),
\begin{align*}
X_{k+1} & := X_k + W_{t_{k+1}}^1 - W_{t_k}^1\m,\\
Y_{k+1} & := \max\bigg(Y_k +  \frac{1}{2}\big(W_{t_k}^1 + W_{t_{k+1}}^1) \big(W_{t_{k+1}}^2 - W_{t_k}^2\big),\\[-2pt]
&\hspace{15.5mm}Y_k + \frac{1}{2}\big(W_{t_k}^1 + W_{t_{k+\frac{1}{2}}}^1\big)W_{t_k\m,\m t_{k+\frac{1}{2}}}^2 + \frac{1}{2}\big(W_{t_{k+\frac{1}{2}}}^1 + W_{t_{k+1}}^1\big)W_{t_{k+\frac{1}{2}}\m,\m t_{k+1}}^2\bigg),
\end{align*}
where $t_k := kh$ with $h = \frac{T}{N}$ denoting the coarsest step size. Then, at time $T$, we have
\begin{align*}
\E\big[y_T\big] = 0,\hspace{5mm}\E\big[Y_N\big] = \frac{1}{8}T.
\end{align*}
\end{theorem}
\begin{proof} Since the $x$-component of the SDE is solved exactly by Heun's method, we have $X_k = W_{t_k}^1$ for all $k\geq 0$. Rearranging the second part of the $y$-component gives
\begin{align*}
& Y_k + \frac{1}{2}\big(W_{t_k}^1 + W_{t_{k+\frac{1}{2}}}^1\big)\big(W_{t_{k+\frac{1}{2}}}^2 - W_{t_k}^2\big) + \frac{1}{2}\big(W_{t_{k+\frac{1}{2}}}^1 + W_{t_{k+1}}^1\big)\big(W_{t_{k+1}}^2 - W_{t_{k+\frac{1}{2}}}^2\big)\\
&\mm = Y_k +  \frac{1}{2}\big(W_{t_k}^1 + W_{t_{k+1}}^1) \big(W_{t_{k+1}}^2 - W_{t_k}^2\big)\\
&\hspace{13.5mm} + \frac{1}{2}\big(W_{t_{k+\frac{1}{2}}}^1 - W_{t_k}^1\big) \big(W_{t_{k+1}}^2 - W_{t_{k+\frac{1}{2}}}^2\big) - \frac{1}{2}\big(W_{t_{k+1}}^1 - W_{t_{k+\frac{1}{2}}}^1\big) \big(W_{t_{k+\frac{1}{2}}}^2 - W_{t_k}^2\big).
\end{align*}
Since the last line gives the difference between one and two steps of Heun's method, it is enough for us to compute its expected value (conditional on being non-negative).
\begin{align*}
\E\big[Y_{k+1}\big] & = \E\bigg[Y_k +  \frac{1}{2}\big(W_{t_k}^1 + W_{t_{k+1}}^1) \big(W_{t_{k+1}}^2 - W_{t_k}^2\big)\\[-2pt]
&\hspace{12.7mm} + \max\Big(0, \frac{1}{2}\Big(W_{t_k\m,\m t_{k+\frac{1}{2}}}^1 W_{t_{k+\frac{1}{2}},\m t_{k+1}}^2 - W_{t_{k+\frac{1}{2}},\m t_{k+1}}^1 W_{t_k\m,\m t_{k+\frac{1}{2}}}^2\Big)\Big)\bigg]\\
& = \E\big[Y_k\big] + \frac{1}{4}\m \E\Big[\big|W_{t_k\m,\m t_{k+\frac{1}{2}}}^1 W_{t_{k+\frac{1}{2}},\m t_{k+1}}^2 - W_{t_{k+\frac{1}{2}},\m t_{k+1}}^1 W_{t_k\m,\m t_{k+\frac{1}{2}}}^2\big|\Big]\\[3pt]
& = \E\big[Y_k\big] + \frac{1}{8}h\m \E\Big[\big|Z_1 Z_2 - Z_3 Z_4\big|\Big],
\end{align*}
where $Z_1\m,Z_2\m, Z_3\m, Z_4\sim N(0,1)$ are independent standard normal random variables.
The result now follows as $\m\E[|Z_1 Z_2 - Z_3 Z_4|] = 1$, which is shown in Theorem \ref{thm:random_matrix_fact}..
\end{proof}

We refer to \cite[Section 4.1]{gaines1997variable} for a counterexample with no-skip dyadic step sizes and the Euler-Maruyama method -- which does not satisfy condition (\ref{eq:intro_ode_condition}) in general.

\section{Numerical example}\label{sect:numerical_example} 

In this section, we will consider the well-known SABR model used in mathematical finance \cite{SABR2017Cai, SABR2018Cui, SABR2002Hagan, SABR2017Leitao}.
The model describes the evolution of an interest (or exchange) rate $S$ whose volatility $\sigma$ is also stochastic. It is given by
\begin{align}\label{eq:SABR}
\begin{split}
dS_t & = \sqrt{1-\rho^2}\m \sigma_t (S_t)^\beta  dW_t^1 + \rho\m \sigma_t (S_t)^\beta  dW_t^2,\\
d\sigma_t & = \alpha\m\sigma_t\m dW_t^2,
\end{split}
\end{align}
where $W$ is a standard two-dimensional Brownian motion and $\alpha,\beta,\rho$ are parameters.
For simplicity, we set $\alpha = 1$ and $\beta = \rho = 0$. In Stratonovich form, (\ref{eq:SABR}) then becomes
\begin{align}\label{eq:SABR_ver2}
\begin{split}
dS_t & = \exp(\nu_t)\circ dW_t^1,\\
d\nu_t & = -\frac{1}{2}\m dt + dW_t^2,
\end{split}
\end{align}
with $\nu_t := \log (\sigma_t)$. We will also use $S_0 = \nu_0 = 0$ and a fixed time horizon of $T=8$.
As discussed in Section \ref{sect:organisation}, it may be difficult to accurately simulate the SDE (\ref{eq:SABR_ver2}) since the noise vector fields are non-commutative. That is, second iterated integrals of Brownian motion (or L\'{e}vy areas) will appear within stochastic Taylor expansions. Moreover, the first noise term in (\ref{eq:SABR_ver2}) is neither bounded or has bounded derivatives.
Whilst this means that the SABR model does not satisfy the regularity assumptions used by our main result, we expect it will make the SDE more challenging to simulate with constant step sizes (and thus hopefully better handled using adaptive step sizes).
\medbreak
By It\^{o}'s isometry, we can show that the local error for Heun's method is given by\vspace*{-1.5mm}
\begin{align*}
\E\bigg[\bigg(S_{t+h} - \hspace{-0.25mm}\bigg(S_t+\frac{1}{2}\sigma_t\big(1 + e^{-\frac{1}{2}h + (W_{t+h}^2-W_t^2)}\big)\big(W_{t+h}^1 - W_t^1\big)\bigg)\hspace{-0.25mm}\bigg)^{\hspace{-0.125mm}2}\,\bigg|\,\sigma_t\bigg] \hspace{-0.05mm} = \hspace{-0.05mm}\frac{1}{4}\sigma_t^2 h (e^h - 1).
\end{align*}
Thus, by setting this proportional to $h$, we arrive at the following previsible step size,\vspace*{-0.5mm}
\begin{align}\label{eq:previsible_step}
h(\nu_t) := \log(1 + C\exp(-2\nu_t)),\\[-18pt] \nonumber
\end{align}
where $C > 0$ is a user-specified constant. In our experiments, we found that (\ref{eq:previsible_step}) gave the best choice of previsible step size -- even for the Euler-Maruyama method. A simple lower bound for its average step size $\E[h(\nu_t)]$ is given in \ref{append:previsible}.\smallbreak

For the non-previsible step size, we have modified the Proportional-Integral (PI) controller for commutative SDEs detailed in \cite{burrage2004variable, ilie2015adaptive} and implemented it in the popular ODE/SDE simulation package ``Diffrax'' \cite{kidger2022nde}. This proposes a new step size $h_{k+1}$ as\vspace*{-1mm}
\begin{align*}
h_{k+1} & := \min\big(\{\widetilde{h}_{k+1}\}\hspace{-0.125mm}\cup\hspace{-0.125mm}\{h > 0 : t_{k+1} + h\text{ corresponds to a previously rejected time}\}\big),\\[2pt]
\widetilde{h}_{k+1} & := h_k\times\\[-2pt]
&\mm \,\,\bigg(\mathrm{Fac}_{\max}\hspace{-0.5mm}\wedge\hspace{-0.5mm}\bigg(\mathrm{Fac}_{\min}\hspace{-0.5mm}\vee\hspace{-0.5mm}\mathrm{Fac}\m\bigg(\frac{C}{e(Y_k, h_k, \mathcal{W}_k)}\bigg)^{K_I}\hspace{-0.5mm}\bigg(\frac{e(Y_{k-1}, h_{k-1}, \mathcal{W}_{k-1})}{e(Y_k, h_k, \mathcal{W}_k)}\bigg)^{K_P}\,\bigg)\bigg),
\end{align*}
where $\{\m\mathrm{Fac}_{\max}\m, \mathrm{Fac}_{\min}\}$ are the maximum and minimum factors $h_k$ can change by, $\mathrm{Fac}\in(0,1)$ is a safety factor, $\{K_I, K_P\}$ are the \textit{integral} and \textit{proportional} coefficients, $C > 0$ is the absolute tolerance and $e(Y_k, h_k, \mathcal{W}_k)$ denotes the rescaled local error estimator $\frac{1}{C\sqrt{d}}\|Y_{k+1} - \widetilde{Y}_{k+1}\|_2$ obtained by computing $(Y_{k+1}\m, \widetilde{Y}_{k+1})$ from $(Y_k, h_k, \mathcal{W}_k)$\footnote{In our case, the pair $(Y_{k+1}\m, \widetilde{Y}_{k+1})$ will be obtained from an embedded Runge-Kutta method. Whilst it performed worse in our experiment, one can also compare one step and two half-steps of the same method. Of course, the numerical solution should then be propagated using the two half-steps.}.
If $e(Y_{k+1}\m, h_{k+1}\m, \mathcal{W}_{k+1}) > 1$, then the step $h_{k+1}$ is rejected and $\m\mathrm{Fac}_{\max}$ is temporarily set equal to $\mathrm{Fac} < 1$ so that the next candidate step size is strictly less than $h_k$. Otherwise, if the estimator satisfies $e(Y_{k+1}\m, h_{k+1}\m, \mathcal{W}_{k+1}) \leq 1$, then $h_{k+1}$ is accepted. 
For PI adaptive SDE solvers, Diffrax always uses an initial candidate step size of $0.01$. We also set the factor parameters to their default values in Diffrax, $\mathrm{Fac}_{\max} = 10$, $\mathrm{Fac}_{\min} = 0.2, \mathrm{Fac} = 0.9$ and, as recommended by \cite{ilie2015adaptive}, we choose $K_I = 0.3, K_P = 0.1$. \medbreak

This approach differs from standard PI controllers as it revisits the time points where Brownian information was generated (instead of ``skipping over'' such times). Otherwise, the ``no-skip'' condition needed for our main result (Theorem \ref{thm:pathwise_conv_thm}) would not be satisfied and convergence issues similar to our counterexample (\ref{eq:counterexample_sde}) may arise. However, due to its implementation within Diffrax, our modified PI controller is not restricted to dyadic times, so does not satisfy the dyadic assumption of Theorem \ref{thm:pathwise_conv_thm}. Just as for the step size control detailed in Definition \ref{def:intro_adaptive_control}, we were unable to prove that $\mesh(\D_n)\rightarrow 0$ for the PI adaptive step sizes. However, similar to the discussion in Remark \ref{rmk:intro_adaptive_control}, we expect that it would be straightforward to show $\mesh(\D_n)\rightarrow 0$ if a uniform lower bound could be established for the local error estimator $\|Y_{k+1} - \widetilde{Y}_{k+1}\|_2\m$.\medbreak
\noindent
In the experiment, we compare the following numerical methods and step size controls:\smallbreak
\begin{itemize}[leftmargin=1.75em]
\item The Euler-Maruyama method, with constant and previsible step sizes.\smallbreak
\item Heun's method (Definition \ref{def:heun_euler}), with constant, previsible and PI step sizes.\smallbreak
\item The SPaRK method (Definition \ref{def:srk_heun}) with constant, previsible and PI step sizes.
\end{itemize}\smallbreak

For each of these methods, we will compute the following strong error estimator:

\begin{definition}[Monte Carlo estimator for strong / root mean squared errors] Let $\{Y_{t_k}^C\}_{0\m\leq\m t_k\m\leq\m T}$ be a numerical solution of SDE (\ref{eq:SABR}), computed over $[0,T]$ with $Y_0^C := y_0\m$, where $C$ denotes the user-specified parameter used to determine step sizes. For $N\geq 1$, we consider $N$ independent samples of $Y^C\hspace{-0.25mm}$, denoted by $\{\{Y_{t_k, i}^C\}_{k\m\geq\m 0}\}_{1\m\leq\m i\m\leq\m N}$.\smallbreak
Then, for a given $C$ and $N$, we will define the strong error estimator $S_{N,C}$ as\vspace{-0.5mm}
\begin{align}\label{eq:strong_error}
S_{N,C} := \max_{0\m\leq\m k\m\leq\m 32}\sqrt{\frac{1}{N}\sum_{i\m =\m 1}^N \Big\|Y_{k\Delta T,\m i}^C - Y_{k\Delta T,\m i}^{\text{fine}}\Big\|_2^2}\,,\\[-20pt]\nonumber
\end{align}
where $\Delta T := \frac{1}{32}T$ and $\{Y_{t_k,\m i}^{\text{fine}}\}_{0\m\leq\m t_k\m\leq\m T}$ denotes a ``fine'' numerical approximation that is computed using the same SDE solver -- but which uses either a fine step size of $h^{\text{fine}}:=2^{-14}T\m$ (if $Y^C$ is computed using constant step sizes) or $C_{\text{fine}} := \frac{1}{8}C_{\min}$ (where $C_{\min}$ is the smallest $C$ used to compute $Y^C$ in the adaptive step size experiments).
\end{definition}
\begin{remark}
A crucial aspect of the error estimator (\ref{eq:strong_error}) is that $Y^C$ and $Y^{\text{fine}}$ are computed with respect to the same Brownian motion. Fortunately, this can be achieved in a convenient manner by Diffrax's (single-seed) Virtual Brownian Tree \cite{jelincic2024VBT}. However, we do note that (\ref{eq:strong_error}) is an estimator for mean squared error -- which is very different from the pathwise convergence established by our main result, Theorem \ref{thm:pathwise_conv_thm}.
\end{remark}

The estimated convergence rates for the different methods and step size controls are given in Figure \ref{fig:sabr}
and Python code for reproducing these results can be found at \href{https://github.com/andyElking/Adaptive_SABR}{github.com/andyElking/Adaptive\_SABR}.
\begin{figure}[!h]
    \centering
    \includegraphics[width=0.75\textwidth]{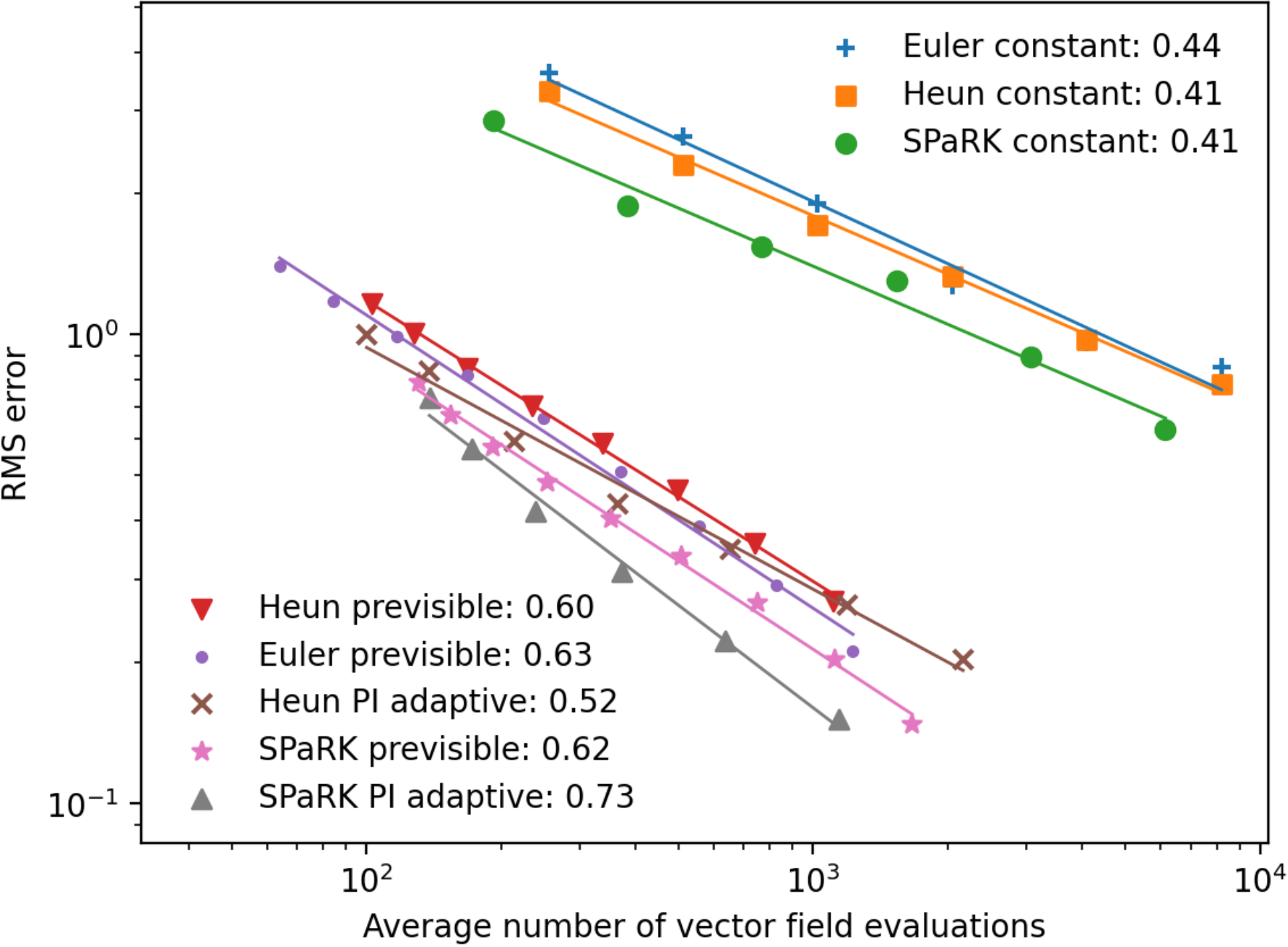}
    \caption{Estimated convergence rates for the different numerical methods and step size controls. Errors for the constant and variable step size SDE solvers were estimated with $N = 100,000$ samples. As the Euler-Maruyama, Heun and SPaRK methods use $1,2$ and $3$ vector field evaluations per step, the x-axis is obtained as the average number of steps multiplied by these numbers.}
    \label{fig:sabr}
\end{figure}\vspace{-1mm}

As expected, all of the SDE solvers exhibit a convergence rate of roughly $O(\sqrt{h})$, with SPaRK being the most accurate method. In particular, with variable step sizes, we see that each numerical method becomes an order of magnitude more accurate.
Perhaps this is not surprising for the previsible step sizes, which were model specific and derived to achieve the desired mean squared error condition. However, for the PI adaptive SDE solvers, we have simply used the recommended parameter settings.
Therefore, based on this numerical experiment, we believe that our modified PI step size controller is a promising technique for simulating general multidimensional SDEs.

\section{Conclusion}\label{sect:conclusion} The primary contribution of this paper shows that standard numerical schemes for SDEs, such as the Heun and no$\m$-area Milstein methods, converge pathwise when used with non-previsible adaptive step sizes (under mild conditions).
This enables us to use embedded methods, which are popular within ODE numerics, to estimate local errors for an adaptive step size controller in the general SDE setting. However, we leave further investigations of such adaptive controllers as a future topic.\medbreak

We also introduce a numerical method called ``SPaRK'' which uses space-time L\'{e}vy area to achieve smaller local errors than standard ``increment only'' SDE solvers.
This stochastic Runge$\m$-Kutta method is compatible with the proposed adaptive step size theory and we empirically demonstrated its efficacy on a two-dimensional SDE.\medbreak

\noindent
The results of this paper also lead to further questions regarding adaptive SDE solvers:\medbreak

\begin{itemize}[leftmargin=1.5em]

\item Do numerical methods that ``weakly'' approximate L\'{e}vy area terms, or otherwise introduce local unbiasaed $O(h)$ perturbations, converge with adaptive step sizes?\medbreak

Examples of such schemes include high order weak approximations \cite{jelincic2025levygan, ninomiya2009ninomiya, ninomiya2008victoir, talay1990HighOrder} and the below extension of Heun's method for general It\^{o} SDEs introduced in \cite{roberts2012euler}:

\begin{definition}[Heun's method with embedded Euler-Maruyama for It\^{o} SDEs]
We define a numerical solution $Y = \{Y_k\}$ for the It\^{o} SDE (\ref{eq:ito_SDE}) by $Y_0 := y_0$ and, for $k\geq 0$, we define $\big(Y_{k+1}, \widetilde{Y}_{k+1}\big)$ similar to Heun's method (see Definition \ref{def:heun_euler}) as
\begin{align}
Z_{k+1} & = Y_k + F\big(Y_k\big)\big(\m\overline{W}_{\n k} + \overline{S}_k\big),\nonumber\\[2pt]
Y_{k+1} & = Y_k + \frac{1}{2}\m F\big(Y_k\big)\big(\m\overline{W}_{\n k} + \overline{S}_k\big) + \frac{1}{2}\m F(Z_{k+1})\big(\m\overline{W}_{\n k} - \overline{S}_k\big),\label{eq:ito_heun}\\[2pt]
\widetilde{Y}_{k+1} & := Y_k + F\big(Y_k\big)\overline{W}_{\n k}\m,\label{eq:ito_heun_embed}
\end{align}
where $\overline{S}_k := \begin{pmatrix} 0 & \sqrt{h_k}\m S_k\end{pmatrix}^\top$ with the random vector $S_k\sim\text{Uniform}\big(\big\{-1,1\big\}^d\m\big)$ assumed to be independent of the Brownian motion $W$. Here, $F$ is given by (\ref{eq:new_notation1}).
\end{definition}\medbreak

\item Can we extend our main result to apply to no-skip partitions that are not dyadic?\medbreak

\item Is it possible to obtain convergence rates for non-previsible adaptive SDE solvers?\medbreak

\item What further improvements can be made to the proposed modified PI controller?
(for example, finding a better approach for choosing the initial candidate step size)\medbreak

\item Does the L\'{e}vy construction detailed in Appendix \ref{sect:levy_contruct} extend to an explicit algorithm for generating $N$ terms in the polynomial expansion of Brownian motion? \cite{foster2020poly, habermann2021poly}\medbreak

This would allow adaptive solvers to use accurate L\'{e}vy area approximations \cite{foster2023levyarea}.\medbreak

\item Do solvers that use non-previsible adaptive step sizes converge for stiff SDEs? \cite{rackauckas2020stiff}

\end{itemize}

\section*{Acknowledgements}

We appreciate the support from the University of Bath, the DataSig group under EPSRC grant EP/S026347/1 and the Alan Turing Institute.
We would also like to thank Terry Lyons and Harald Oberhauser for supervising the first author as a graduate student and having countless interesting discussions about stochastic numerics, rough paths and (non-previsible) adaptive step sizes for SDEs.\vspace{-2mm}

\newpage

\bibliographystyle{siamplain}
\bibliography{references}

\newpage

\appendix

\section{Improved local accuracy of the SPaRK method (\ref{eq:srk})}\label{append:taylor_expansions}In this section, we will detail certain properties of the newly introduced stochastic Runge$\m$-Kutta method (\ref{eq:srk}) that make it an appealing alternative to Heun's method.\medbreak

We first recall the definition of this method (using the notation from Section \ref{sect:examples_of_methods}),
\begin{align}
Y_{k+\frac{1}{2}} := Y_k & + F\big(Y_k\big)\bigg(\frac{1}{2}\overline{W}_{\n k} + \sqrt{3}\,\overline{H}_k\bigg),\hspace{7.5mm}
Z_{k+1} := Y_k +  F\big(Y_{k+\frac{1}{2}}\big)\overline{W}_{\n k}\m,\nonumber\\[3pt]
Y_{k+1} := Y_k & + F\big(Y_k\big)\big(a\m\overline{W}_{\n k} + \overline{H}_k\big) + b\m F\big(Y_{k+\frac{1}{2}}\big)\overline{W}_{\n k} + F\big(Z_{k+1}\big)\big(a\m\overline{W}_{\n k} - \overline{H}_k\big),\label{eq:append_srk}
\end{align}
with $a = \frac{3-\sqrt{3}}{6}$ and $b = \frac{\sqrt{3}}{3}$. In Section \ref{sect:methods_as_cdes}, the scheme (\ref{eq:append_srk}) was Taylor expanded as
\begin{align}\label{eq:srk_expansion}
Y_{k+1} = Y_k & + f(Y_k)h_k + g(Y_k)W_k\\
& + g^{\m\prime}(Y_k)\m g(Y_k)\bigg(\frac{1}{2}W_k^{\otimes 2} + H_k\otimes W_k - W_k\otimes H_k\bigg) + o(h_k).\nonumber
\end{align}
By comparing (\ref{eq:srk_expansion}) with the Taylor expansion of the SDE solution started from $Y_k\m$, we see that the leading term in the expansion of the one-step mean squared error is 
\begin{align*}
&\E\Bigg[\bigg\|g^{\m\prime}(Y_k)\m g(Y_k)\bigg(\int_{t_k}^{t_{k+1}}W_{t_k\m, t}\otimes \circ\, dW_t - \Big(\frac{1}{2}W_k^{\otimes 2} + H_k\otimes W_k - W_k\otimes H_k\Big)\bigg)\bigg\|_2^2\Bigg]\\
& = \sum_{i\neq j}\E\Big[\big\|g_j^{\m\prime}(Y_k)\m g_i(Y_k)\big\|_2^2\Big]\underbrace{\E\Bigg[\bigg(\int_{t_k}^{t_{k+1}}\hspace*{-1mm}W_{t_k\m, t}^i \circ dW_t^j - \Big(\frac{1}{2}W_k^i W_k^j + H_k^i W_k^j - W_k^i H_k^j\Big)\bigg)^2\Bigg]}_{=\frac{1}{12} h_k^2\, \text{ by \cite[Theorem 5.4]{foster2023levyarea} and Brownian scaling}}\\
& = \frac{1}{12} h_k^2\sum_{i\neq j}\E\Big[\big\|g_j^{\m\prime}(Y_k)\m g_i(Y_k)\big\|_2^2\Big].
\end{align*}
On the other hand, two half steps of Heun's method on $[t_k\m, t_{k+1}]$ can be expanded as
\begin{align*}
Y_{k+1} & = Y_{k+\frac{1}{2}} + \frac{1}{2}f(Y_{k+\frac{1}{2}})h_k + g(Y_{k+\frac{1}{2}})W_{k+\frac{1}{2}} + g^{\m\prime}(Y_{k+\frac{1}{2}})\m g(Y_{k+\frac{1}{2}})\bigg(\frac{1}{2}W_{k+\frac{1}{2}}^{\otimes 2}\bigg) + o(h_k)\\[3pt]
& = Y_k + f(Y_k)h_k + g(Y_k)\big(W_k + W_{k+\frac{1}{2}}\big)\\
&\hspace{8.5mm} + g^{\m\prime}(Y_k)\m g(Y_k)\bigg(\frac{1}{2}W_k^{\otimes 2} + W_k\otimes W_{k+\frac{1}{2}} + \frac{1}{2}W_{k+\frac{1}{2}}^{\otimes 2}\bigg) + o(h_k),
\end{align*}
where we now use the notation $W_k := W_{t_{k+\frac{1}{2}}} - W_{t_k}$ and $W_{k+\frac{1}{2}} := W_{t_{k+1}} - W_{t_{k+\frac{1}{2}}}$.
For $i\neq j$, we have $\int_s^t W_{s,u}^i \circ dW_u^j = \int_s^t W_{s,u}^i dW_u^j\m$, and so applying It\^{o}'s isometry gives
\begin{align*}
&\E\Bigg[\bigg(\int_{t_k}^{t_{k+1}}\hspace*{-1mm}W_{t_k\m, t}^i \circ dW_t^j - \Big(\frac{1}{2}W_k^i W_k^j + W_k^i W_{k+\frac{1}{2}}^j + \frac{1}{2}W_{k+\frac{1}{2}}^i W_{k+\frac{1}{2}}^j\Big)\bigg)^{\vspace{-1mm}2}\Bigg]\\
&= \E\Bigg[\bigg(\int_{t_k}^{t_{k+\frac{1}{2}}}\hspace*{-1mm}W_{t_k, t}^i \circ dW_t^j - \frac{1}{2}W_k^i\m W_k^j\bigg)^2\hspace{-0.75mm} + \bigg(\int_{t_{k+\frac{1}{2}}}^{t_{k+1}}\hspace*{-1mm}W_{t_{k+\frac{1}{2}}, t}^i \circ dW_t^j -\frac{1}{2}W_{k+\frac{1}{2}}^i W_{k+\frac{1}{2}}^j\bigg)^{\vspace{-1mm}2}\m \Bigg]\\
& = \frac{1}{4}\bigg(\frac{1}{2}h_k\bigg)^2 + \frac{1}{4}\bigg(\frac{1}{2}h_k\bigg)^2 = \frac{1}{8}\m h_k^2\m.
\end{align*}
Hence, the leading term in the mean squared error of the two-step Heun's method is
\begin{align*}
\frac{1}{8} h_k^2\sum_{i\neq j}\E\Big[\big\|g_j^{\m\prime}(Y_k)\m g_i(Y_k)\big\|_2^2\Big].
\end{align*}
Thus, provided the step size is sufficiently small, we see that one step of the stochastic Runge$\m$-Kutta method (\ref{eq:srk}) is more accurate than two steps of Heun's method (\ref{eq:heun}).
In addition, two steps of Heun's method uses an extra evaluation of the vector fields.\medbreak

It is also known that Heun's method achieves second order weak convergence for SDEs with additive noise \cite{leimkuhler2014OLD} (that is, when $g(y) = \sigma$ for a constant matrix $\sigma\in\R^{e\times d}$),
\begin{align*}
    dy_t = f(y_t)\m dt + \sigma\m dW_t\m,
\end{align*}
where, for the purposes of this section, we will assume that $f$ is sufficiently smooth.
For additive noise SDEs, the Splitting Path Runge-Kutta (SPaRK) method becomes
\begin{align*}
Y_{k+\frac{1}{2}} := Y_k & + \frac{1}{2}\m f(Y_k)\m h_k + \sigma\bigg(\frac{1}{2}W_k + \sqrt{3}\m H_k\bigg), \mm
Z_{k+1} := Y_k  + f(Y_{k+\frac{1}{2}})h_k + \sigma\m W_k\m,\\[3pt]
Y_{k+1} := Y_k & + \bigg(\frac{3-\sqrt{3}}{6}\m f(Y_k) + \frac{\sqrt{3}}{3}\m f(Y_{k+\frac{1}{2}}) + \frac{3-\sqrt{3}}{6}\m f(Z_{k+1})\bigg)h_k + \sigma W_k\m.
\end{align*}
Using the classical Taylor theorem $f(b) = f(a) + f^\prime(a)(b-a) + \frac{1}{2}f^{\prime\prime}(a)(b-a)^{\otimes 2} + R$ with $R := \frac{1}{2}\int_0^1 (1-t)^2 f^{\prime\prime\prime}(a + t(b-a))(b-a)^{\otimes 3}dt$, we can Taylor expand $Y_{k+1}$ to give 
\begin{align*}
Y_{k+1} = Y_k & + f(Y_k)\m h_k + \sigma W_k + \frac{\sqrt{3}}{3}f^\prime (Y_k)\bigg(\frac{1}{2}\m f(Y_k)\m h_k + \sigma\bigg(\frac{1}{2}W_k + \sqrt{3}\m H_k\bigg)\bigg)h_k\\
& + \frac{3-\sqrt{3}}{6}\m f^\prime(Y_k)\Big(f(Y_{k+\frac{1}{2}})h_k +\sigma W_k\Big)h_k\\
& + \frac{\sqrt{3}}{3}\bigg(\frac{1}{2}f^{\prime\prime} (Y_k)\bigg(\frac{1}{2}\m f(Y_k)\m h_k + \sigma\bigg(\frac{1}{2}W_k + \sqrt{3}\m H_k\bigg)\bigg)^{\otimes 2}\,\bigg)h_k\\
& + \frac{3-\sqrt{3}}{6}\bigg(\frac{1}{2}f^{\prime\prime} (Y_k)\Big( f(Y_{k+\frac{1}{2}})\m h_k + \sigma\m W_k\Big)^{\otimes 2}\,\bigg)h_k + O\big(h_k^\frac{5}{2}\big)\\[3pt]
= Y_k & + f(Y_k)\m h_k + \sigma W_k + f^\prime(Y_k)\m\sigma\bigg(\frac{1}{2} W_k + H_k\bigg)h_k + \frac{1}{2}\m f^\prime(Y_k)f(Y_k)h_k^2 + O\big(h_k^\frac{5}{2}\big)\\
&  + \frac{1}{2}\m f^{\prime\prime}(Y_k)\m \sigma^{\otimes 2}\bigg(\frac{6-\sqrt{3}}{12}\m W_k^{\otimes 2} + \frac{1}{2}\m W_k\otimes H_k + \frac{1}{2}\m H_k\otimes W_k + \sqrt{3}\m H_k^{\otimes 2}\bigg)h_k\m.
\end{align*}
We note that in the above Taylor expansion, we have $\big(\frac{1}{2} W_k + H_k\big)h_k = \int_{t_k}^{t_{k+1}} W_{t_k, t}\m dt$ and $\m\E\Big[\Big(\frac{6-\sqrt{3}}{12}\m W_k^{\otimes 2} + \frac{1}{2}\m W_k\otimes H_k + \frac{1}{2}\m H_k\otimes W_k + \sqrt{3}\m H_k^{\otimes 2}\Big)h_k\Big] = \frac{1}{2}h_k^2 = \E\Big[\int_{t_k}^{t_{k+1}} W_{t_k, t}^{\otimes 2}\m dt\Big]$.\medbreak

Thus, by the standard mean squared error analysis of Milstein and Tretyakov \cite{milstein2004physics},
we see that stochastic Runge$\m$-Kutta method will converge strongly with rate $O(h^\frac{3}{2})$.
Similarly, it is also clear that the method will achieve second order weak convergence.
We refer the reader to \cite{foster2024splitting} for more details on the analysis of commutative noise SDEs. For non-additive noise SDEs, it is straightforward to see that the proposed stochastic Runge$\m$-Kutta method exhibits the same strong convergence rates as Heun's method.\medbreak

To conclude, we summarise the properties of these numerical methods in Table \ref{table:heun_vs_srk}.\vspace{-0.5mm}
\begin{table}[H]
  \begin{center}
  \begin{tabular}{ccccc}
    \toprule
     &
      \multicolumn{2}{c}{Heun's method}  &  
      \multicolumn{2}{c}{Splitting Path}\\
      & One step & Two steps &  \multicolumn{2}{c}{Runge$\m$-Kutta}\\
 & & & & \\[-12pt]
   \midrule \\[-10pt]
    Vector field evaluations & 2 & 4 & \multicolumn{2}{c}{3} \\[4pt]
    Gaussian random vectors & 1 & 2  & \multicolumn{2}{c}{2}  \\[4pt]
 Coefficient in front of leading term &  & \\
 $h_k^2\sum_{i\neq j}\E\Big[\big\|g_j^{\m\prime}(Y_k)\m g_i(Y_k)\big\|_2^2\Big]$ & $\displaystyle\frac{1}{4}$ & $\displaystyle\frac{1}{8}$ &  \multicolumn{2}{c}{$\displaystyle\frac{1}{12}$}\\[6pt]
 in expansion of mean squared error &  & \\
 & & & & \\[-9pt]
   \midrule \\[-9pt]
   Convergence rates for SDEs with... & Strong & Weak  & Strong\hspace{3mm} & Weak \\[4pt]
   general noise &  $O(h^\frac{1}{2})$ & $O(h)$  & $O(h^\frac{1}{2})$\hspace{3mm} & $O(h)$\\[4pt]
   commutative noise & \multirow{2}{*}{$O(h)$} & \multirow{2}{*}{$O(h)$} &  \multirow{2}{*}{$O(h)$}\hspace{3mm} & \multirow{2}{*}{$O(h)$}\\
   $\big(g_i^\prime(y)g_j(y) = g_j^\prime(y)g_i(y)\big)$ & & & & \\[4pt]
   additive noise &  \multirow{2}{*}{$O(h)$} & \multirow{2}{*}{$O(h^2)$} &  \multirow{2}{*}{$O(h^\frac{3}{2})$}\hspace{3mm} & \multirow{2}{*}{$O(h^2)$}\\
    $\big(g(y) \equiv\sigma\in\R^{e\times d}\big)$  & & & & \\[2pt]
    \bottomrule
  \end{tabular}\medbreak
  \caption{Properties of Heun's method (\ref{eq:heun}) and the proposed SPaRK method (\ref{eq:srk}).}\label{table:heun_vs_srk}
  \end{center}
\end{table}\vspace{-5mm}

\section{Generating increments and integrals of Brownian motion}\label{sect:levy_contruct}
In this section, we will briefly outline how increments $W_{s,t}$ and space-time L\'{e}vy areas $H_{s,t} := \frac{1}{t-s}\int_s^t W_{s,u}\m du\m$ of Brownian motion can be generated for adaptive SDE solvers.\medbreak

We emphasise that the results detailed in this section are already well-established in the SDE numerics literature \cite{foster2020thesis, jelincic2024VBT} and are only included here for completeness.\medbreak

To begin, we note that for any $u\in[s,t]$, a path increment $W_{s,u} := W_u - W_s$ can be generated conditional on $W_{s,t}$ using standard properties of the Brownian bridge. When $u = \frac{1}{2}(s+t)$, this leads to well-known \textit{L\'{e}vy construction of Brownian motion}.

\begin{theorem}[Conditional distribution of Brownian increments]
For $s\leq u\leq t$,
\begin{align*}
W_{s,u} & = \frac{u-s}{t-s}\m W_{s,t} + B_{s,u}\m,\mm
W_{u,t} = \frac{t-u}{t-s}\m W_{s,t} - B_{s,u}\m,
\end{align*}
where the Brownian bridge increment $B_{s,u} \sim\mathcal{N}\big(0, \frac{(u-s)(t-u)}{t-s} I_d\big)$ is independent of $W_{s,t} \sim\mathcal{N}\big(0, (t-s) I_d\big)$. When $u = \frac{1}{2}(s+t)$ is the midpoint of $[s,t]$, this simplifies to
\begin{align*}
W_{s,u} & = \frac{1}{2}\m W_{s,t} + B_{s,u}\m,\mm
W_{u,t} = \frac{1}{2}\m W_{s,t} - B_{s,u}\m,
\end{align*}
where $B_{s,u} \sim\mathcal{N}\big(0,  \frac{1}{4}\m (t-s) I_d\big)$.
\end{theorem}
\begin{proof}
Let $B_{s,u} := W_{s,u} - \frac{u-s}{t-s}\m W_{s,t}$. The covariance between $B_{s,u}$ and $W_{s,t}$ is
\begin{align*}
\cov(B_{s,u}\m, W_{s,t}) & = \cov(W_{s,u}\m, W_{s,t}) - \frac{u-s}{t-s}\m \cov(W_{s,t}\m, W_{s,t})\\[-1pt]
&  = (u-s) - \frac{u-s}{t-s}\m(t-s) = 0.
\end{align*}
Since $B_{s,u}$ and $W_{s,t}$ depend on the same Brownian motion, they are jointly normal,
and therefore independent. The variance of $B_{s,u}$ is straightforward to compute.
\end{proof}

Extending the theorem given above to include both increments and space-time L\'{e}vy areas of the Brownian motion is more involved, but based on a similar argument (which also uses that uncorrelated jointly normal random variables are independent).

\begin{theorem}[Conditional distribution of increments and space-time L\'{e}vy areas of Brownian motion within an interval, {\cite[Theorems 6.1.4 and 6.1.6]{foster2020thesis}}]\label{thm:levy_with_areas}For $s\leq u\leq t$,
\begin{align*}
W_{s,u} & = \frac{u-s}{h}\m W_{s,t} + \frac{6(u-s)(t-u)}{h^2}\m H_{s,t} + \frac{2(a+b)}{h}\m X_1\m,\\[3pt]
W_{u,t} & = \frac{t-u}{h}\m W_{s,t} - \frac{6(u-s)(t-u)}{h^2}\m H_{s,t} - \frac{2(a+b)}{h}\m X_1\m,\\[3pt]
H_{s,u} & = \frac{(u-s)^2}{h^2}\m H_{s,t} - \frac{a}{u-s}\m X_1 + \frac{c}{u-s}\m X_2\m,\\[3pt]
H_{u,t} & = \frac{(u-s)^2}{h^2}\m H_{s,t}  - \frac{b}{t-u}\m X_1 - \frac{c}{t-u}\m X_2\m,
\end{align*}
where $W_{s,t}\sim\mathcal{N}(0,h I_d)\m, H_{s,t}\sim\mathcal{N}\big(0, \frac{1}{12}\m h I_d\big)$ and $X_1\m, X_2\sim\mathcal{N}(0, I_d)$ are independent, $h = t-s$ denotes the length of the interval and the coefficients $a,b,c\in\R$ are given by
\begin{align*}
a & := \frac{(u-s)^\frac{7}{2}(t-u)^\frac{1}{2}}{2h\sqrt{(u-s)^3 + (t-u)^3}}\m,\mm
b := \frac{(u-s)^\frac{1}{2}(t-u)^\frac{7}{2}}{2h\sqrt{(u-s)^3 + (t-u)^3}}\m,\\[3pt]
 &\hspace{22.5mm} c := \frac{\sqrt{3}(u-s)^\frac{3}{2}(t-u)^\frac{3}{2}}{6\sqrt{(u-s)^3 + (t-u)^3}}\m.
\end{align*}
When $u = \frac{1}{2}(s+t)$ is the midpoint of the interval $[s,t]$, these formulae simplify to
\begin{align*}
W_{s,u} & = \frac{1}{2}\m W_{s,t} + \frac{3}{2}\m H_{s,t} + Z_{s,u}\m,\mm
W_{u,t} = \frac{1}{2}\m W_{s,t} - \frac{3}{2}\m H_{s,t} - Z_{s,u}\m,\\[3pt]
H_{s,u} & = \frac{1}{4}\m H_{s,t} - \frac{1}{2}\m Z_{s,u} + \frac{1}{2}\m N_{s,t}\m,\mm
H_{u,t} = \frac{1}{4}\m H_{s,t} - \frac{1}{2}\m Z_{s,u} - \frac{1}{2}\m N_{s,t}\m,
\end{align*}
where $W_{s,t}\sim\mathcal{N}(0,h I_d), Z_{s,u}\sim\mathcal{N}\big(0, \frac{1}{16}\m h I_d\big)$ and $H_{s,t}\m, N_{s,t}\sim\mathcal{N}\big(0, \frac{1}{12}\m h I_d\big)$ are all independent Gaussian random vectors.
\end{theorem}\vspace*{-7.5mm}
\begin{figure}[!hbt]
\begin{center}
    \includegraphics[width=0.8\textwidth]{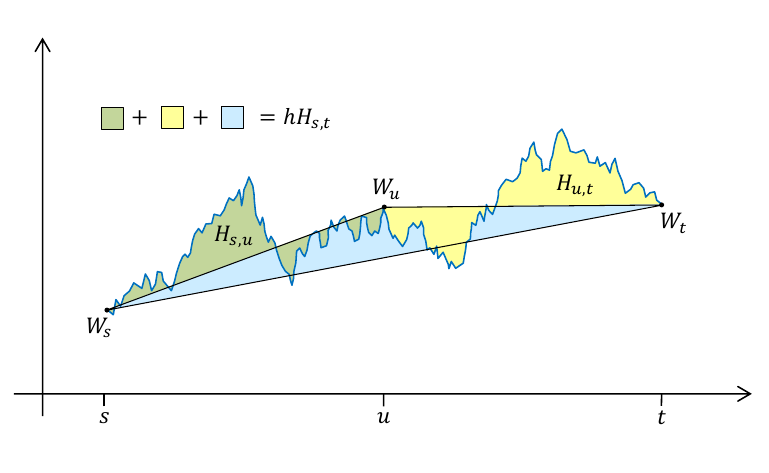}\vspace*{-7.5mm}
    \caption{Illustration of the Brownian increments and areas in Theorem \ref{thm:levy_with_areas} (diagram from \cite{foster2020thesis})}
    \label{fig:levy_construct_with_areas}
\end{center}
\end{figure}

\clearpage
\newpage

\section{Norm of the determinant for a Gaussian random matrix}

In this section, we present the following known result, which was helpful in Section \ref{sect:counterexample}.
\begin{theorem}[$L^1(\mathbb{P})$ norm of the determinant for a Gaussian random matrix]
For $n\geq 1$, let $A_n = \{a_{ij}\}_{1\m\leq\m i,\m j\m\leq\m n}$ be an $n\times n$ matrix with independent entries $a_{ij}\sim\mathcal{N}(0,1)$ and let $Z\sim\mathcal{N}(0,1)$ denote a standard normal random variable. Then,
\begin{align}\label{eq:magic_determinant}
\E\big[\m|\det A_n\m|\m\big] & = \E\big[\m|Z|^n\big].
\end{align}
\end{theorem}\medbreak
\begin{proof}
Let $p_n$ denote the probability density function of $|\det A_n\m|$. Then the Mellin transform of $p_n$ is given by \cite[Equation (2.3)]{cicuta2000matrix}, but with $a=\frac{1}{2}$ and $\beta = 1$. That is,
\begin{align*}
\mathcal{M}_n (s) := \int_0^\infty x^{s-1} p_n(x)\m dx = \Big(\frac{1}{2}\Big)^{-\frac{1}{2}n(s-1)}\prod_{j=1}^n\bigg(\frac{\Gamma\big(\frac{1}{2}(s-1) + \frac{1}{2}j\big)}{\Gamma\big(\frac{1}{2}j\big)}\bigg),
\end{align*}
where $\Gamma$ is the standard Gamma function. Since $\E\big[\m|\det A_n\m|\m\big] = \int_0^\infty x\m p_n(x)\m dx$, we have
\begin{align*}
\E\big[\m|\det A_n\m|\m\big] & = \Big(\frac{1}{2}\Big)^{-\frac{1}{2}n}\prod_{j=1}^n\bigg(\frac{\Gamma\big(\frac{1}{2} + \frac{1}{2}j\big)}{\Gamma\big(\frac{1}{2}j\big)}\bigg)\\
& = \Big(\frac{1}{2}\Big)^{-\frac{1}{2}n}\,\frac{\Gamma\big(\frac{1}{2} + \frac{1}{2}n\big)}{\Gamma\big(\frac{1}{2}\big)}\\
& = 2^{\frac{1}{2}n} \,\frac{\Gamma\big(\frac{1}{2}(n+1)\big)}{\sqrt{\pi}},
\end{align*}
which is the $n$-th moment of a half-normal distribution (see \cite[Equation (18)]{winkelbauer2014moments}).
\end{proof}\medbreak

In particular, when $n=2$, we obtain the result used in the proof of Theorem \ref{thm:counterexample}.

\begin{theorem} Let $A,B,C,D\sim\mathcal{N}(0,1)$ be independent random variables.
Then\label{thm:random_matrix_fact}
\begin{align}
\E\big[\m|AD - BC\m|\m\big] = 1.
\end{align}
\end{theorem}\medbreak

\begin{remark}
Although (\ref{eq:magic_determinant}) immediately follows from \cite[Equation (2.3)]{cicuta2000matrix} and  \cite[Equation (18)]{winkelbauer2014moments}, it was difficult for the authors to find it stated in the literature.
Hence, (\ref{eq:magic_determinant}) may be a lesser-known (but elegant) fact about the normal distribution.
\end{remark}\medbreak

\section{A lower bound for the SABR-specific previsible step size}
By the construction of the previsible step size (\ref{eq:previsible_step}), applying Heun's method to the SABR model (\ref{eq:SABR_ver2}) with step size $h\equiv h(\nu_t)$ gives a local mean squared error of $\frac{1}{4}Ch$. Since the local errors propagate linearly, the global mean squared error will be $O(C)$.
On the other hand, we can also consider the ``average step size'' taken by the method.\label{append:previsible}
\begin{theorem} Let $\nu_t := -\frac{1}{2}t + W_t$ denote the log-volatility component of (\ref{eq:SABR_ver2}) and let $h(\nu_t) := \log(1+Ce^{-2\nu_t})$ be the previsible step size control with $C>0$. Then
\begin{align}\label{eq:h_lower_bound}
\E\big[h(\nu_t)\big] \geq \log(1+C e^t).
\end{align}
\end{theorem}
\begin{proof}
We first note that the previsible step size control $x\mapsto h(x)$ is convex as
$\frac{d^2}{dx^2} h(x) = \frac{d^2}{dx^2}(\log(1+Ce^{-2x})) = \frac{4C e^{2x}}{(C + e^{2x})^2} > 0$.  Thus, by Jensen's inequality, we have
\begin{align*}
\E\big[h(\nu_t)\big] \geq h\big(\E\big[\nu_t\big]\big) = h\Big(\hspace{-0.25mm}-\frac{1}{2}t\Big) = \log\big(1+C e^t\big).
\end{align*}
\end{proof}

Hence, if $C$ is sufficiently small, the average step size at any $t\in [0,T]$ is $O(C)$, and results in a global $L^2(\mathbb{P})$ error of $O(\sqrt{C}\m)$. 
This mimics the $O(\sqrt{h}\m)$ strong convergence of SDE solvers with a constant step size. However, since (\ref{eq:h_lower_bound}) increases with time, we can see that the variable step size methods tend to take increasingly larger steps.

\end{document}